\documentclass[11pt,reqno]{amsart}
\usepackage{amsmath, amssymb, amsthm,amsfonts}
\usepackage{enumitem}
\usepackage{graphicx,color}
\usepackage{graphics}
\usepackage{comment}
\usepackage[OT2,T1]{fontenc}
\usepackage{tikz}
\DeclareSymbolFont{cyrletters}{OT2}{wncyr}{m}{n}
\DeclareMathSymbol{\Sha}{\mathalpha}{cyrletters}{"58}

\DeclareMathOperator{\SL}{SL}
\DeclareMathOperator{\DSL}{\widetilde{SL}}
\DeclareMathOperator{\GL}{GL}
\DeclareMathOperator{\PGL}{PGL}
\DeclareMathOperator{\Sh}{Sh}

\def\ov{\overline}

\newcommand{\Hup}{\mathbb{H}}

\newcommand{\ord}{{\rm ord}}

\newcommand{\A}{\mathbb{A}}

\newcommand{\Qs}{\mathbb{Q}^{\times}}

\newcommand{\Q}{{\mathbb Q}}
\newcommand{\Z}{{\mathbb Z}}

\newcommand{\C}{{\mathbb C}}
\newcommand{\R}{{\mathbb R}}

\newcommand{\kro}[2]{\left( \frac{#1}{#2} \right) }

\newcommand{\hs}[1]{\left( #1 \right)_p}
\newcommand{\hst}[1]{\left( #1 \right)_2}

\newcommand{\hsbt}[1]{\Big( #1 \Big)_2}
\newcommand{\mat}[4]{\left(\begin{matrix} #1 & #2 \\ #3 & #4 \end{matrix}\right)}
\newcommand{\smallmat}[4]{\left(\begin{smallmatrix} #1 & #2 \\ #3 & #4 \end{smallmatrix}\right)}

            \DeclareFontFamily{U}{wncy}{} 
            \DeclareFontShape{U}{wncy}{m}{n}{%
               <5>wncyr5%
               <6>wncyr6%
               <7>wncyr7%
               <8>wncyr8%
               <9>wncyr9%
               <10>wncyr10%
               <11>wncyr10%
               <12>wncyr6%
               <14>wncyr7%
               <17>wncyr8%
               <20>wncyr10%
               <25>wncyr10}{} 
\DeclareMathAlphabet{\cyr}{U}{wncy}{m}{n}

\begin{document}

\newtheorem{thm}{Theorem}
\newtheorem*{thms}{Theorem}
\newtheorem{lem}{Lemma}[section]
\newtheorem{prop}[lem]{Proposition}
\newtheorem{alg}[lem]{Algorithm}
\newtheorem{cor}[lem]{Corollary}
\newtheorem{conj}[lem]{Conjecture}
\newtheorem{remark}{Remark}

\theoremstyle{definition}

\newtheorem{ex}{Example}

\theoremstyle{remark}

\title[Minus space of level $8M$]{Newforms of half-integral weight: the minus space of 
$S_{k+1/2}(\Gamma_0(8M))$
}

\author{Ehud Moshe Baruch}
\address{Department of Mathematics\\
         Technion\\
         Haifa , 32000\\
         Israel}
\email{embaruch@math.technion.ac.il}
\thanks{}
\author{Soma Purkait}
\address{Tokyo University of Science, Noda\\
         Japan}
\email{somapurkait@gmail.com}
\keywords{Hecke algebras, Half-integral weight forms, Ueda-Niwa isomorphism, 
Newforms}
\subjclass[2010]{Primary: 11F37; Secondary: 11F12, 11F70}
\date{July 2017}

\begin{abstract}
We compute generators and relations for a certain 
$2$-adic Hecke algebra of level $8$ associated
with the double cover of $\SL_2$ and a $2$-adic 
Hecke algebra of level $4$ associated with $\PGL_2$. 
We show that these two Hecke algebras are isomorphic as 
expected from the Shimura correspondence.
We use the $2$-adic generators to define classical Hecke operators 
on the space of holomorphic modular forms of weight 
$k+1/2$ and level $8M$ where $M$ is odd and square-free. 
Using these operators and our previous results
on half-integral weight forms of level $4M$ we define a subspace 
of the space of half-integral weight forms as a
common $-1$ eigenspace of certain Hecke operators. 
Using the relations and a result of Ueda we show that this
subspace which we call the minus space is isomorphic as a 
Hecke module under the Ueda correspondence to
the space of new forms of weight $2k$ and level $4M$. 
We observe that the forms in the minus space satisfy a Fourier coefficient
condition that gives the complement of the plus space but 
does not define the minus space.
\end{abstract}
\maketitle

\section{Introduction}
It has been observed by Waldspurger \cite{Waldspurger} that the theory of 
half-integral weight modular forms is more complicated for levels of the forms 
$8M$ and $16M$ where $M$ is an odd integer. In particular, one of the 
Waldspurger's hypothesis, (H2), for his celebrated 
formula for central twisted L-values for an integral weight newform requires 
the newform to be either not supercuspidal at $2$ or has level divisible by $16$, 
that is, the corresponding half-integral weight ``newform'' should have 
level $4M$ or level divisible by $32$.
Ueda in an unpublished work \cite{Ueda-un} observed that the 
levels $2^k M$ where $2\leq k \leq 6$ pose greater 
difficulties than general $k$. In this work we complete the work of Ueda to
obtain a decomposition of the space of holomorphic forms of weight $k+1/2$
and level $8M$ where $M$ is odd and square-free. An attempt at such theory has 
been made by Manickam, Meher, Ramakrishnan \cite{M-M-R} but our results and 
methods differ completely. For more details on this difference see 
Remark~\ref{rem:MMR} in Section~\ref{sec:minus}. 

Kohnen \cite{Kohnen1, Kohnen2} defines the 
plus space $S_{k+1/2}^{+}(4M)$ to be a subspace of 
$S_{k+1/2}(\Gamma_0(4M))$ which is the positive eigenspace of a certain
Hecke operator. In the case $M=1$, Loke and Savin~\cite{L-S} gave an 
interpretation of the Kohnen plus space in representation theory language 
using a $2$-adic Hecke algebra of level $4$.
Manickam, Ramakrishnan, Vasudevan \cite{M-R-V} defined a new subspace 
inside this space as an orthogonal complement of certain subspaces. 
In \cite{B-PII} we used Hecke algebra 
approach to show that the new subspace is given as a common $-1$ eigenspace 
of certain Hecke operators. In this paper we extend this approach by considering 
$2$-adic Hecke algebras of level $8$ and describe it using generators 
and relations. 
We obtain an isomorphism between certain $2$-adic Hecke algebras which is 
part of the Shimura correspondence. We translate $2$-adic Hecke elements 
into classical Hecke operators that satisfies the relations coming from the $2$-adic elements. Once we have 
these operators and relations we combine it with our results for the case 
$4M$ to define the minus space of level $8M$ as a common $-1$ eigenspace of 
certain classical Hecke operators. Using the Ueda Hecke isomorphism between 
$S_{k+1/2}(\Gamma_0(8M))$ and $S_{2k}(\Gamma_0(4M))$ we show that the minus space 
at level $8M$ is Hecke isomprphic to the subspace of new forms in 
$S_{2k}(\Gamma_0(4M))$ and satisfies the strong-multiplicity one. 
Further if 
$f = \sum_{n=1}^{\infty} a_n q^n$ is in the minus space at level $8M$ then 
$a_n=0$ for $(-1)^k n \equiv 0, 1 \pmod{4}$. This condition is 
exactly opposite to the Kohnen's plus space Fourier coefficient condition 
and is contrary to the 
assertion of \cite{M-M-R}.

\section{Preliminaries and Notation}
We will be following the notation of our paper~\cite{B-PII}.
Let $k,\ N$ denote positive integers with $4$ dividing $N$.
Let $\mathcal{G}$ be the set of all ordered pairs $(\alpha, \phi(z))$ where 
$\alpha = \smallmat{a}{b}{c}{d} \in \GL_2(\R)^{+}$ and $\phi(z)$ is a holomorphic 
function on the upper half plane $\Hup$ such that 
$\phi(z)^2 = t\det(\alpha)^{-1/2}(cz+d)$ with $t$ in the unit circle 
$S^1$. 
For $\zeta = (\alpha, \phi(z)) \in \mathcal{G}$ define the slash operator 
$|[\zeta]_{k+1/2}$ 
on functions $f$ 
on $\Hup$ 
by $f|[\zeta]_{k+1/2}(z) = f(\alpha z)(\phi(z))^{-2k-1}$. 
For an even Dirichlet character $\chi$ modulo $N$, let
\[\Delta_0(N, \chi):=\{\alpha^* = (\alpha, j(\alpha,z)) \in 
\mathcal{G}\ |\ \alpha \in \Gamma_0(N)\} \leq \mathcal{G}.\] 
where $j(\alpha,z) = \chi(d)\varepsilon_d^{-1}\kro{c}{d}(cz+d)^{1/2}$,
here $\varepsilon_d = 1$ or $i$ according as 
$d \equiv 1\text{ or }3 \pmod{4}$ and $\kro{c}{d}$ is as 
in Shimura's notation \cite{Shimura}. The space of holomorphic cusp forms
$S_{k +1/2}(\Gamma_0(N), \chi)$ satisfy 
$f|[\alpha^*]_{k +1/2}(z) =f(z)$ for all $\alpha \in \Delta_0(N, \chi)$ and 
we have well-known family of Hecke operators $\{T_{p^2}\}_{p \nmid N},\ \{U_{p^2}\}_{p\mid N}$
on $S_{k +1/2}(\Gamma_0(N), \chi)$ (please refer to \cite{Shimura} for details).

Let $\DSL_2(\Q_p)$ be the non-trivial central extension of $\SL_2(\Q_p)$ 
by $\mu_2 = \{\pm1\}$ given by Kubota-Gelbart $2$-cocycle defined below. 

For $g =\mat{a}{b}{c}{d} \in \SL_2(\Q_p)$, define
\[\tau(g) = \begin{cases}
         c & \text{if $c \ne 0$}\\
         d & \text{if $c = 0$}
        \end{cases}; \]
if $p=\infty$, set $s_p(g)=1$ while for a finite prime $p$
\[s_p(g) = \begin{cases}
         \hs{c,d} & \text{if $cd \ne 0$ and $\ord_p(c)$ is odd}\\
         1 & \text{else}.
        \end{cases}
\]
Define the $2$-cocycle $\sigma_p$ on $\SL_2(\Q_p)$ as follows:
\[\sigma_p(g, h) = \hs{\tau(gh)\tau(g),\tau(gh)\tau(h)}s_p(g)s_p(h)s_p(gh).\]
Then the double cover 
$\DSL_2(\Q_p)$ is the set $\SL_2(\Q_p) \times \mu_2$ with the group law:
\[(g,\ \epsilon_1)(h,\ \epsilon_2) = (gh,\ \epsilon_1\epsilon_2\sigma_p(g, h)).\]
For any subgroup $H$ of $\SL_2(\Q_p)$, we shall denote by $\ov{H}$ the complete 
inverse image of $H$ in $\DSL_2(\Q_p)$. 

We consider the following subgroups of $\SL_2(\Z_p)$:
\[K_0^p(p^n) =\left\{ \mat{a}{b}{c}{d} \in \SL_2(\Z_p)\ :\ c \in p^n\Z_p \right\},\]
\[K_1^p(p^n) =\left\{ \mat{a}{b}{c}{d} \in K_0^p(p^n)\ :\  a \equiv 1 \pmod{p^n\Z_p} \right\}.\]
 
By \cite[Proposition 2.8]{Gelbart}, 
$\DSL_2(\Q_p)$ splits over $\SL_2(\Z_p)$ for odd primes $p$, while 
$\DSL_2(\Q_2)$ does not split over $\SL_2(\Z_2)$ and instead splits 
over the subgroup $K_1^2(4)$. It follows that the 
center $M_2$ of $\DSL_2(\Q_2)$ is a cyclic group of order $4$ 
generated by $(-I,1)$.

For an open compact subgroup $S$ of $\DSL_2(\Q_p)$ and 
a genuine character $\gamma$ of $S$, let
$H(S, \gamma)$ be the subalgebra of $C_c^{\infty}(\DSL_2(\Q_p))$ defined by \\
\[\{ f \in C_c^{\infty}(\DSL_2(\Q_p)) : 
f(\tilde{k}\tilde{g}\tilde{k'})=\overline{\gamma}(\tilde{k})\overline{\gamma}(\tilde{k'})f(\tilde{g}) 
\text{ for } \tilde{g} \in \DSL_2(\Q_p),\ \tilde{k},\ \tilde{k'} \in S\}\]
under the usual convolution, 
\[ f_1*f_2(\tilde{h}) = 
\int_{\DSL_2(\Q_p)} f_1(\tilde{g})f_2(\tilde{g}^{-1}\tilde{h}) d\tilde{g} = 
\int_{\DSL_2(\Q_p)} f_1(\tilde{h}\tilde{g})f_2(\tilde{g}^{-1}) d\tilde{g}, \]
Loke and Savin~\cite{L-S} described $H(\ov{K_0^2(4)}, \gamma)$ for a certain 
$\gamma$ of order $4$ using generators 
and relations. We \cite{B-PII} described $H(\ov{K_0^p(p)}, \gamma)$ for odd 
primes $p$ and certain quadratic characters $\gamma$. We translated the elements of local Hecke algebra to obtain 
classical operators $\widetilde{Q}_p$, $\widetilde{Q}'_p$, 
$\widetilde{W}_{p^2}$ on 
$S_{k +1/2}(\Gamma_0(2^nM))$ for $p$ strictly dividing $M$, $M$ odd and 
operators $\widetilde{Q}_2$, $\widetilde{Q}'_2$, 
$\widetilde{W}_{4}$ on 
$S_{k +1/2}(\Gamma_0(4M))$. We shall be using these operators and their properties
in Section~\ref{sec:minus} of this paper. 

We set up a few more notation.
For $s \in \Q_2$, $t \in \Qs_2$ let us define the following elements of 
$\SL_2(\Q_2)$:
\[x(s) =\mat{1}{s}{0}{1},\ y(s)=\mat{1}{0}{s}{1},\  
w(t) = \mat{0}{t}{-t^{-1}}{0},\  h(t)=\mat{t}{0}{0}{t^{-1}}.\]
Let $N=\{(x(s), \epsilon): s \in \Q_2,\ \epsilon = \pm 1\}$, 
$\bar{N}=\{(y(s), \epsilon): s \in \Q_2,\ \epsilon = \pm 1\}$ 
and $T=\{(h(t), \epsilon): t \in \Qs_2,\ \epsilon = \pm 1\}$
be the subgroups of $\DSL_2(\Q_2)$. Then the normalizer
$N_{\DSL_2(\Q_2)}(T)$ of $T$ in $\DSL_2(\Q_2)$ consists of 
elements $(h(t), \epsilon)$, $(w(t), \epsilon)$ for $t \in \Qs_2$.

\section{Hecke Algebra of $\DSL_2(\Q_2)$ modulo $\ov{K_0^2(8)}$}
In this section we shall be describing the local Hecke algebra 
$H(S, \gamma)$ for $S=\ov{K_0^2(8)}$ and certain genuine characters 
$\gamma$ described below.

Let $\gamma$ be a genuine central character given by a sending 
$(-I,1)$ to a primitive fourth root of unity. 
We extend $\gamma$ to $\ov{K_0^2(8)}$ as follows. Extend $\gamma$ to 
$K_1^2(8) \times M_2$ so that it is trivial on $K_1^2(8)$. Since 
$\ov{K_0^2(8)}/(K_1^2(8) \times M_2)$ is generated by $(h(5),1)$ 
it is enough to define it on $(h(5),1)$. 
Since $(h(5),1)$ has order $2$ there are two choices for $\gamma((h(5),1))$. 
We will call the resulting two characters of $\ov{K_0^2(8)}$ as 
$\chi_1$ and $\chi_2$:
\[\chi_1((h(u),1)) = \begin{cases}
                    1 & \text{ if $u \equiv 1,\ 5 \pmod{8\Z_2}$}\\
                    \gamma((-I,1)) & \text{ if $u \equiv 3,\ 7 \pmod{8\Z_2}$,}
                   \end{cases}, \]
\[\chi_2((h(u),1)) = \begin{cases}
                    1 & \text{ if $u \equiv 1 \pmod{8\Z_2}$}\\
                    \gamma((-I,1)) & \text{ if $u \equiv 7 \pmod{8\Z_2}$}\\
                    -1 & \text{ if $u \equiv 5 \pmod{8\Z_2}$}\\
                    -\gamma((-I,1)) & \text{ if $u \equiv 3 \pmod{8\Z_2}$.}
                   \end{cases}\]
Note that for $\gamma = \chi_1,\ \chi_2$ and 
$(A, \epsilon) = (x(s),1)(h(u),1)(y(t),1)(I, \epsilon\delta) 
\in \ov{K_0^2(8)}$ (\cite[Lemma A.4]{B-PII}) we have
\begin{equation}\label{eq:td}
\gamma(A, \epsilon) = \gamma(x(s),1)\cdot \gamma(h(u),1) 
\cdot \gamma(y(t),1) \cdot \gamma(I,\epsilon\delta) =  
\gamma(h(u),1)\cdot \gamma(I,\epsilon\delta).
\end{equation}

We shall from now on denote $K_0^2(8)$ by simply $K_0$. 
Also for $g \in \DSL_2(\Q_2)$ we put $\ov{g}:=(g,1)\in \DSL_2(\Q_2)$.
We shall describe the Hecke algebra $H(\ov{K}_0, \gamma)$ using generators and 
relations.

We have the following proposition. The proof is a routine calculation.
\begin{prop}\label{prop:1}
A complete set of representatives for the double cosets of $\DSL_2(\Q_2)$ modulo 
$\ov{K}_0$ consists of $\ov{g}$ where $g$ varies over the 
following elements of $SL_2(\Q_2)$:
$h(2^n),\ w(2^n)$ for $n \in \Z$, $h(2^n)y(4), h(2^n)y(2)$ for $n \ge 0$, 
$y(4)h(2^{-n}), y(2)h(2^{-n}), w(2^{-n})y(2), y(2)w(2^{-n}), y(2)w(2^{-n})y(2)$ 
for $n \ge 1$, $w(2^{-n})y(4),\ y(4)w(2^{-n}) ,\ y(4)w(2^{-n})y(4),\ 
y(2)w(2^{-n})y(4), y(4)w(2^{-n})y(2)$ for $n \ge 2$ and $y(2)w(2^{-1})y(6)$.

\end{prop}
We will now compute the support of $H(\ov{K}_0, \chi_1)$ and $H(\ov{K}_0, \chi_2)$.
We first have the following lemma on vanishing. 
\begin{lem}
The Hecke algebra $H(\ov{K}_0, \chi_1)$ vanishes on the double cosets of $\ov{K}_0$ 
represented by $\ov{y}(2)$, $\ov{y}(2)\ov{w}(2^{-n})$, $\ov{w}(2^{-n})\ov{y}(2)$, 
$\ov{y}(2)\ov{w}(2^{-n})\ov{y}(2)$, $\ov{h}(2^n)\ov{y}(2)$, $\ov{y}(2)\ov{h}(2^{-n})$, $\ov{y}(2)\ov{w}(2^{-1})\ov{y}(6)$ for $n \ge 1$ and 
$\ov{y}(2)\ov{w}(2^{-n})\ov{y}(4)$, $\ov{y}(4)\ov{w}(2^{-n})\ov{y}(2)$ for 
$n \ge 2$.

The Hecke algebra $H(\ov{K}_0, \chi_2)$ vanishes on the double cosets of $\ov{K}_0$ represented by 
$\ov{y}(4)$, $\ov{y}(4)\ov{h}(2^{-n})$, $\ov{h}(2^{n})\ov{y}(4)$ for $n \ge 1$, 
$\ov{y}(2)\ov{w}(2^{-1})\ov{y}(6)$, and $\ov{y}(2)\ov{w}(2^{-n})\ov{y}(4)$, 
$\ov{y}(4)\ov{w}(2^{-n})\ov{y}(2)$, $\ov{w}(2^{-n})\ov{y}(4)$, 
$\ov{y}(4)\ov{w}(2^{-n})$, $\ov{y}(4)\ov{w}(2^{-n})\ov{y}(4)$ where $n \ge 2$.
\end{lem}
\begin{proof}
Recall that (\cite[Lemma3.1]{B-PII})
$H(\ov{K}_0, \gamma)$ is supported on $\tilde{g}$ 
if and only if for every $\tilde{k} \in 
K_{\tilde{g}}:=\ov{K}_0 \cap \tilde{g} \ov{K}_0 \tilde{g}^{-1}$ we have 
$\gamma([\tilde{k}^{-1},\ \tilde{g}^{-1}]) = 1$. So 
to check the vanishing at $\tilde{g}$ we need to just find suitable 
$\tilde{k}$.

For example, for 
$A=y(2)w(2^{-n})$, take $B= \mat{-3}{2}{-8}{5}$. Then
 $\ov{B} \in K_{\ov{A}} $ and 
 \[[\ov{B}^{-1},\ \ov{A}^{-1}] = \left(\mat{5+2^{2n+2}}{-2}{8+3\cdot2^{2n+1}}{-3},\ -1\right);\]
The above commutator is
of the form $\left(\mat{-3}{*}{0}{-3}\pmod{8\Z_2},\ -1\right)$ 
and in its triangular decomposition (as in equation~\ref{eq:td}) 
$\delta=1$. Since $\chi_1$ takes value $-1$, the vanishing of 
$H(\ov{K}_0, \chi_1)$ follows on the double coset of $A$. 
The vanishing of $H(\ov{K}_0, \chi_1),\ H(\ov{K}_0, \chi_2)$ at 
the double cosets listed in the lemma follow similarly. 
\end{proof}
  
\begin{lem}\label{lem:sup1}
$H(\ov{K}_0, \chi_1)$ and $H(\ov{K}_0, \chi_2)$ are supported on the
double cosets of $\ov{K}_0$ represented by $\ov{h}(2^n)$ and $\ov{w}(2^{-n})$ 
for $n\in \Z$.

$H(\ov{K}_0, \chi_1)$ is supported on $\ov{y}(4)$ while 
$H(\ov{K}_0, \chi_2)$ is supported on $\ov{y}(2)$.
\end{lem}
\begin{proof}
The proof is similar to the proof of \cite[Lemma 3.5]{B-PII}
and mainly uses \cite[Lemma 3.1, Lemma 3.2 and Lemma A.3]{B-PII}.

We illustrate the case of support of $H(\ov{K}_0, \chi_2)$ on $\ov{y}(2)$
The rest involve similar calculations.

We note that 
\[K_{\ov{y}(2)} = \left\{ (\mat{a-2b}{b}{c+2(a-d)-4b}{2b+d}, \pm1) : 
\mat{a}{b}{c}{d} \in K_0,\ \ord_2(b) \ge 1\right\}.\]
has a triangular decomposition 
$K_{\tilde{A}} = N^{K_{\tilde{A}}} T^{K_{\tilde{A}}} \bar{N}^{K_{\tilde{A}}}$
where
\[N^{K_{\ov{y}(2)}} = \{(x(s),\pm1): \ord_2(b) \ge 1\}, \quad 
T^{K_{\ov{y}(2)}} = T^{\ov{K}_0}, \quad 
\bar{N}^{K_{\ov{y}(2)}} = \bar{N}^{\ov{K}_0}.\]
For $B= x(s)$ where $\ord_2(s) \ge 1$ (assume $s \ne 0$) we have 
$B^{-1}A^{-1}BA = \mat{1+2s+4s^2}{2s^2}{-4s}{-2s+1}$
and $s_2(B^{-1}A^{-1}BA)=\hst{-s,2s-1}$ when $\ord_2(s)$ is odd, $1$ else. 
Thus for $\ord_2(s)\ge 2$ we have $s_2(B^{-1}A^{-1}BA)=1$. If $s =2u$ with $u$ a unit, 
\[\hst{-s, -2s+1} = \hst{-2u, -4u+1} = \hst{-2, -4u+1}\hst{-u, -4u+1}\]
\[=\hst{-2, -3} = -1.\]
As before, since $\ord_2(s)\ge 1$ the $\delta$-factor in the 
triangular decomposition of $[(B,\ \epsilon)^{-1},\ \ov{y}(2)]$ is $1$ and so 
\[\chi_2([(B,\ \epsilon)^{-1},\ \ov{y}(2)]) = \begin{cases}
                                               \chi_2((h(5),-1)) = -1\times-1 =1
                                               & \text{if $\ord_2(s) =1$}\\
                                               1 & \text{if $\ord_2(s) \ge 2$}.
                                              \end{cases}
\] 
For $B= h(u)$ and $B=y(t)$ in $K_{\ov{y}(2)}$ we check that 
$[(B,\ \epsilon)^{-1},\ \ov{y}(4)] \in K_1^2(8) \times \{1\}$
and so $\chi_2$ takes value $1$.
\end{proof}

We note down the following decomposition of double cosets into union of single cosets 
that would be useful in obtaining generators and relations amongst the Hecke algebra 
elements.
\begin{lem}\label{lem:decomp1}
\begin{small}
\begin{enumerate}
\item[(a)] For $n \ge 0$,
\[K_0h(2^n)K_0 = \bigcup_{s\in \Z_2/2^{2n}\Z_2} x(s)h(2^n)K_0 = \bigcup_{s\in \Z_2/2^{2n}\Z_2} K_0h(2^n)y(8s).\]
\item[(b)] For $n \ge 1$,
\[K_0h(2^{-n})K_0 = \bigcup_{s\in \Z_2/2^{2n}\Z_2} y(8s)h(2^{-n})K_0 = \bigcup_{s\in \Z_2/2^{2n}\Z_2} K_0h(2^{-n})x(s).\]
\item[(c)] For $n \ge 2$,
\[K_0w(2^{-n})K_0 = \bigcup_{s\in \Z_2/2^{2n-3}\Z_2} y(8s)w(2^{-n})K_0 = \bigcup_{s\in \Z_2/2^{2n-3}\Z_2} K_0w(2^{-n})y(8s).\] 
\item[(d)] 
 \[K_0w(2^{-1})K_0 = \bigcup_{s\in \Z_2/2\Z_2} x(s)w(2^{-1})K_0 = \bigcup_{s\in \Z_2/2\Z_2} K_0w(2^{-1})x(s).\]
 \item[(e)] For $n \ge 0$,
\[K_0w(2^{n})K_0 = \bigcup_{s\in \Z_2/2^{2n+3}\Z_2} x(s)w(2^{n})K_0 = \bigcup_{s\in \Z_2/2^{2n+3}\Z_2} K_0w(2^{n})x(s).\]
\item[(f)] $K_0y(4)K_0 = K_0y(4) = y(4)K_0$. 
\item[(g)] For $n \ge 1$,
\[K_0h(2^n)y(4)K_0 = \bigcup_{s\in \Z_2/2^{2n}\Z_2} x(s)h(2^n)y(4)K_0 = \bigcup_{s\in \Z_2/2^{2n}\Z_2} K_0h(2^n)y(4+8s).\]
\item[(h)] For $n \ge 1$,
 \[K_0y(4)h(2^{-n})K_0 = \bigcup_{s\in \Z_2/2^{2n}\Z_2} y(4+8s)h(2^{-n})K_0 = \bigcup_{s\in \Z_2/2^{2n}\Z_2} K_0y(4)h(2^{-n})x(s).\]
\item[(i)] For $n \ge 2$,
\[K_0w(2^{-n})y(4)K_0 = \bigcup_{s\in \Z_2/2^{2n-3}\Z_2} y(8s)w(2^{-n})y(4)K_0 = \bigcup_{s\in \Z_2/2^{2n-3}\Z_2} K_0w(2^{-n})y(4+8s).\] 
\item[(j)] For $n \ge 2$,
\[K_0y(4)w(2^{-n})K_0 = \bigcup_{s\in \Z_2/2^{2n-3}\Z_2} y(4+8s)w(2^{-n})K_0 = \bigcup_{s\in \Z_2/2^{2n-3}\Z_2} K_0y(4)w(2^{-n})y(8s).\]
\item[(k)] For $n \ge 2$,
\[K_0y(4)w(2^{-n})y(4)K_0 = 
\bigcup_{s\in \Z_2/2^{2n-3}\Z_2} y(4+8s)w(2^{-n})y(4)K_0 = 
\bigcup_{s\in \Z_2/2^{2n-3}\Z_2} K_0y(4)w(2^{-n})y(4+8s).\]
\item[(l)]\[K_0 y(2) K_0 = \bigcup_{s \in \Z_2/2\Z_2} x(s)y(2)K_0 = \bigcup_{s \in \Z_2/2\Z_2} K_0y(2)x(s).\]
\item[(m)] \[K_0 y(2)w(2^{-1}) K_0 = \bigcup_{s \in \Z_2/2\Z_2} x(s)y(2)w(2^{-1})K_0 = \bigcup_{s \in \Z_2/2\Z_2} K_0y(2)w(2^{-1})x(s).\]
\item[(n)] \[K_0 w(2^{-1})y(2) K_0 = \bigcup_{s \in \Z_2/2\Z_2} x(s)w(2^{-1})y(2)K_0 = \bigcup_{s \in \Z_2/2\Z_2} K_0w(2^{-1})y(2)x(s).\]
\item[(o)] \[K_0 y(2)w(2^{-1})y(2) K_0 = K_0 y(2)w(2^{-1})y(2) = y(2)w(2^{-1})y(2) K_0.\]
\item[(p)] For $n\ge 2$, 
\[K_0 y(2)w(2^{-n}) K_0 = \bigcup_{s \in \Z_2/2^{2n-2}\Z_2} y(2+4s)w(2^{-n})K_0=
\bigcup_{s \in \Z_2/2^{2n-2}\Z_2} K_0y(2)w(2^{-n})y(8s).\]
\item[(q)] For $n\ge 2$, 
\[K_0 w(2^{-n})y(2) K_0 = \bigcup_{s \in \Z_2/2^{2n-2}\Z_2} y(8s)w(2^{-n})y(2)K_0=
\bigcup_{s \in \Z_2/2^{2n-2}\Z_2}K_0w(2^{-n})y(2+4s).\]
\item[(r)] For $n \ge 1$,
\begin{equation*}
\begin{split}
K_0h(2^n)y(2)K_0 &= \bigcup_{s\in \Z_2/2^{2n+1}\Z_2} x(s)h(2^n)y(2)K_0 \\
&=
\bigcup_{s \in \Z_2/2^{2n}\Z_2} K_0h(2^n)y(2+4s) 
\bigcup_{s \in \Z_2/2^{2n}\Z_2} K_0y(4)h(2^n)y(2+4s).
\end{split}
\end{equation*}
\item[(s)] For $n \ge 1$,
\begin{equation*}
\begin{split}
K_0y(2)h(2^{-n})K_0 &= \bigcup_{s\in \Z_2/2^{2n+1}\Z_2} K_0y(2)h(2^{-n})x(s) \\
&=
\bigcup_{s \in \Z_2/2^{2n}\Z_2}y(2+4s)h(2^{-n})K_0 
 \bigcup_{s \in \Z_2/2^{2n}\Z_2}y(2+4s)h(2^{-n})y(4)K_0.
\end{split}
\end{equation*}
\item[(t)] For $n \ge 2$,
\begin{equation*}
\begin{split}
& K_0y(2)w(2^{-n})y(2)K_0 \\
&= \bigcup_{s\in \Z_2/2^{2n-2}\Z_2} y(2+4s)w(2^{-n})y(2)K_0 \bigcup_{s\in \Z_2/2^{2n-2}\Z_2} y(2+4s)w(2^{-n})y(6)K_0\\
&= \bigcup_{s\in \Z_2/2^{2n-2}\Z_2} K_0y(2)w(2^{-n})y(2+4s) \bigcup_{s\in \Z_2/2^{2n-2}\Z_2} K_0y(6)w(2^{-n})y(2+4s). 
\end{split}
\end{equation*}
\end{enumerate}
\end{small}
\end{lem}

Using Lemma~\ref{lem:sup1}, \cite[Lemma 3.4]{B-PII} and 
the above decomposition we have the following corollary.
\begin{cor}\label{cor:sup2}
$H(\ov{K}_0, \chi_1)$ is
supported on $\ov{y}(4)\ov{h}(2^{-n})$, 
$\ov{h}(2^n)\ov{y}(4)$ for $n\ge 1$ and 
$\ov{y}(4)\ov{w}(2^{-n})$, $\ov{w}(2^{-n})\ov{y}(4)$, 
$\ov{y}(4)\ov{w}(2^{-n})\ov{y}(4)$ for $n \ge 2$.

$H(\ov{K}_0, \chi_2)$ is
supported on $\ov{y}(2)\ov{w}(2^{-n})$, $\ov{w}(2^{-n})\ov{y}(2)$ 
for $n \ge 2$ and $\ov{h}(2^{n})\ov{y}(2)$, $\ov{y}(2)\ov{h}(2^{-n})$ 
for $n \ge 1$.
\end{cor}

Note that we can not use the argument in the proof of the above corollary to 
show support of $H(\ov{K}_0, \chi_2)$ on the double cosets of
$\ov{w}(2^{-1})\ov{y}(2)$, $\ov{y}(2)\ov{w}(2^{-1})$ 
and $\ov{y}(2)\ov{w}(2^{-n})\ov{y}(2)$ for $n\ge 1$. 
In these cases we check the support directly as in Lemma~\ref{lem:sup1}. 

Thus we have the following proposition.
\begin{prop}\label{prop:sup1}
$H(\ov{K_0^2(8)}, \chi_1)$ is supported on precisely the double cosets of 
$\ov{K_0^2(8)}$ represented by 
\[\{\ov{h}(2^n),\ \ov{w}(2^{-n})\}_{n\in \Z} \cup \ov{y}(4) \cup \{\ov{h}(2^n)\ov{y}(4),\ \ov{y}(4)\ov{h}(2^{-n})\}_{n \ge 1} \cup \]
\[\{\ov{y}(4)\ov{w}(2^{-n}),\ \ov{w}(2^{-n})\ov{y}(4),\ \ov{y}(4)\ov{w}(2^{-n})\ov{y}(4)\}_{n\ge 2}.\]

$H(\ov{K_0^2(8)}, \chi_2)$ is supported on precisely the double cosets of $\ov{K_0^2(8)}$ 
represented by 
\[\{\ov{h}(2^n),\ \ov{w}(2^{-n})\}_{n\in \Z}\ \cup\ \ov{y}(2)\ \cup \]
\[\{\ov{y}(2)\ov{w}(2^{-n}),\ \ov{w}(2^{-n})\ov{y}(2),\ \ov{y}(2)\ov{w}(2^{-n})\ov{y}(2),\ \ov{h}(2^n)\ov{y}(2),\ \ov{y}(2)\ov{h}(2^{-n})\}_{n \ge 1}.\]
\end{prop}

\subsection{Generators and Relations}

Let $\gamma$ be either $\chi_1$ or $\chi_2$.
Following Loke and Savin~\cite{L-S} we extend the character 
$\gamma$ on $M_2$ to the normalizer subgroup 
$N_{\DSL_2(\Q_2)}(T)$ of torus $T$ in $\DSL_2(\Q_2)$ 
by defining $\gamma(\ov{h}(2^n))=1$ for all $n\in \Z$ and 
$\gamma(\ov{w}(1))= \frac{1+ \gamma((-I,1))}{\sqrt{2}}=:\zeta_8$, 
a $8$-th root of unity.  

For $n\in \Z$, define the elements $\mathcal{T}_n$ and
$\mathcal{U}_n$ of $H(\ov{K_0^2(8)}, \gamma)$ supported respectively on 
the $\ov{K_0^2(8)}$ double cosets of $(h(2^n),1)$ and 
$(w(2^{-n}),1)$ such that
\begin{equation}\label{eq:defn}
\mathcal{T}_n(\tilde{k}(h(2^n),1)\tilde{k'}) = 
\overline{\gamma}(\tilde{k})\overline{\gamma}((h(2^n),1))
\overline{\gamma}(\tilde{k'}),
\end{equation}
\[\mathcal{U}_n(\tilde{k}(w(2^{-n}),1)\tilde{k'}) = 
\overline{\gamma}(\tilde{k})\overline{\gamma}((w(2^{-n}),1))\overline{\gamma}(\tilde{k'})
\quad \text{for $\tilde{k}$, $\tilde{k'} \in \ov{K_0^2(8)}$}.\]
We use the decomposition Lemma~\ref{lem:decomp1} and \cite[Lemma 3.4]{B-PII} to 
obtain the following relations in $H(\ov{K_0^2(8)}, \gamma)$.
\begin{lem}\label{lem:rel3}
\begin{enumerate} 
 \item If $mn \ge 0$ then $\mathcal{T}_m*\mathcal{T}_n = \mathcal{T}_{m+n}$.
 \item For $n \le 0$, $\mathcal{U}_1 * \mathcal{T}_n = \mathcal{U}_{1+n}$ and 
 for $n \ge 0$, $\mathcal{T}_n * \mathcal{U}_1 = \mathcal{U}_{1-n}$.
 \item For $n \ge 0$, $\mathcal{U}_2 * \mathcal{T}_n = \mathcal{U}_{2+n}$ and 
 for $n \le 0$, $\mathcal{T}_n * \mathcal{U}_2 = \mathcal{U}_{2-n}$.
 \item For $m \ge 2$, $\mathcal{U}_1 * \mathcal{U}_m = \mathcal{T}_{m-1}$ and
 $\mathcal{U}_m * \mathcal{U}_1 = \mathcal{T}_{1-m}$.
 \item For $m \le 1$, $\mathcal{U}_2 * \mathcal{U}_m = \mathcal{T}_{m-2}$ and
 $\mathcal{U}_m * \mathcal{U}_2 = \mathcal{T}_{2-m}$.
\end{enumerate}
\end{lem}

\subsection{The algebra $H(\ov{K_0^2(8)}, \chi_1)$}
Consider the case when $\gamma = \chi_1$. Since $H(\ov{K_0}, \chi_1)$ 
is supported on $\ov{K}_0\ov{y}(4)\ov{K}_0$, we define $\mathcal{V}$ to be an element 
of $H(\ov{K_0}, \chi_1)$ that is supported precisely on $\ov{K}_0\ov{y}(4)\ov{K}_0$ 
such that $\mathcal{V}(\ov{y}(4))=1$. 
Now since $\mu(\ov{y}(4))\mu(\ov{y}(4))=\mu(\ov{y}(8))$, 
using Lemma~\cite[Lemma 3.4]{B-PII} we get that 
$\mathcal{V}* \mathcal{V}$ is supported precisely on 
$\ov{K}_0\ov{y}(8)\ov{K}_0 = \ov{K}_0$ and 
$\mathcal{V} * \mathcal{V} ((I,1)) = 
\mathcal{V} * \mathcal{V}(\ov{y}(8))= 
\mathcal{V}(\ov{y}(4))\mathcal{V}(\ov{y}(4)) = 1$, so we get that 
$\mathcal{V} * \mathcal{V} =1$.

Similarly, $\mathcal{U}_1 * \mathcal{V}$ 
is supported precisely at $\ov{K}_0\ov{w}(2^{-1})\ov{y}(4)\ov{K}_0$ and 
it's value at $\ov{w}(2^{-1})\ov{y}(4)$ is $\mathcal{U}_1(\ov{w}(2^{-1}))$ 
as $\mathcal{V}(\ov{y}(4)=1$. 
But note that $\ov{K}_0\ov{w}(2^{-1})\ov{y}(4)\ov{K}_0 = 
\ov{K}_0 \ov{w}(2^{-1})\ov{K}_0$, in fact 
\[\ov{w}(2^{-1})\ov{y}(4)= (\mat{9}{-1}{-8}{1},1)\ov{w}(2^{-1})(\mat{1}{-2}{0}{1},1),\] 
so 
\[\mathcal{U}_1(\ov{w}(2^{-1})) = 
\mathcal{U}_1 * \mathcal{V}(\ov{w}(2^{-1})\ov{y}(4))= \]
\[ =\mathcal{U}_1 * \mathcal{V} 
((\mat{9}{-1}{-8}{1},1)\ov{w}(2^{-1})(\mat{1}{-2}{0}{1},1)) = 
\mathcal{U}_1 * \mathcal{V}(\ov{w}(2^{-1})),\]
and thus $\mathcal{U}_1 * \mathcal{V} = \mathcal{U}_1$. 
Similarly we get $\mathcal{V} * \mathcal{U}_1 = \mathcal{U}_1$.
\begin{lem}\label{lem:rel4}
For $\mathcal{V},\ \mathcal{U}_1\in H(\ov{K_0}, \chi_1)$ we have following relations:
\begin{enumerate}
 \item $\mathcal{V} * \mathcal{V} =1$.
 \item $\mathcal{U}_1 * \mathcal{V} = \mathcal{U}_1 = \mathcal{V} * \mathcal{U}_1$.
\end{enumerate}
\end{lem}

\begin{prop}\label{prop:rel5}
\begin{enumerate}
\item $\mathcal{U}_2 * \mathcal{U}_2 = 2$.
\item $\mathcal{U}_1 * \mathcal{U}_1 = 2 + 2\mathcal{V}$.
\item $\mathcal{U}_2 * \mathcal{V} * \mathcal{U}_2 
= \sqrt{2}\ \mathcal{V} * \mathcal{U}_2 * \mathcal{V}$.
\item $\mathcal{U}_0 * \mathcal{U}_0 = 8 + 2\sqrt{2}\ \mathcal{U}_0 + 8\mathcal{V}$.
\item $\mathcal{U}_0 * \mathcal{V} = \mathcal{U}_0 = \mathcal{V} * \mathcal{U}_0$
and consequently,  
$\frac{\mathcal{U}_0}{\sqrt{2}}
*(\frac{\mathcal{U}_0}{\sqrt{2}}-4)*(\frac{\mathcal{U}_0}{\sqrt{2}}+2)=0$.
\end{enumerate}
\end{prop}

We shall use the following version of \cite[Lemma 3.3]{B-PII}.
\begin{lem}\label{lem:rel1'}
Let $f_1,\ f_2 \in H(\gamma)$ and $f_1$ is supported on 
$\ov{K}_{0}\tilde{x}\ov{K}_{0}=\bigcup_{i=1}^{m}\ov{K}_0\tilde{\alpha}_i$
and 
$f_2$ is supported on 
$\ov{K}_{0}\tilde{y}\ov{K}_{0}$ and 
let $\ov{K}_{0}\tilde{y}^{-1}\ov{K}_{0}= \bigsqcup_{j=1}^n\tilde{\beta}_j\ov{K}_0$. Then 
\[f_1 * f_2 (\tilde{g}) = \sum_{j=1}^{n}f_1(\tilde{g}\tilde{\beta}_j)
f_2(\tilde{\beta}_j^{-1})\]
and the non-zero summands are the ones for which there exists a 
$j$ such that $\tilde{g}\tilde{\beta}_j \in K_0\tilde{\alpha}_i$.
\end{lem}

\begin{proof}[Proof of Proposition~\ref{prop:rel5}]
We shall prove $(3)$. The proof of $(1)$ and $(2)$ are similar.
Let $\mathcal{W}_2 = \mathcal{U}_2 * \mathcal{V}$ and 
$\mathcal{Z}_2 =\mathcal{V}* \mathcal{U}_2 * \mathcal{V}$. 
Then using the decomposition in Lemma~\ref{lem:decomp1} and \cite[Lemma 3.4]{B-PII}
we see that $\mathcal{W}_2,\ \mathcal{Z}_2$ are respectively supported on 
$\ov{K}_0\ov{w}(2^{-2})\ov{y}(4)\ov{K}_0$ and 
$\ov{K}_0\ov{y}(4)\ov{w}(2^{-2})\ov{y}(4)\ov{K}_0$ and 
$\mathcal{W}_2(\ov{w}(2^{-2})\ov{y}(4)) = \ov{\gamma}(\ov{w}(1)) = 
\mathcal{Z}_2(\ov{y}(4)\ov{w}(2^{-2})\ov{y}(4))$.

So to get the identity we will first compute the support of 
$\mathcal{W}_2 * \mathcal{U}_2$.

Using Lemma~\ref{lem:decomp1}, 
$\ov{K}_0 \ov{w}(2^{-2})\ov{y}(4)\ov{K}_0 = 
\bigsqcup_{s=0,1}\ov{K}_0\tilde{\alpha}_s$ and 
$\ov{K}_0 \ov{w}(2^{-2})^{-1}\ov{K}_0= 
\bigsqcup_{t=0,1}\tilde{\beta}_t\ov{K}_0$ where
\[\tilde{\alpha}_s=\ov{w}(2^{-2})\ov{y}(4+8s), \quad \tilde{\beta}_t= \ov{y}(-8t)\ov{w}(-2^{-2}),\]
The matrix part of $\tilde{\beta}_t\tilde{\alpha}_s^{-1}$ are 
\[\mat{-1}{-1/4}{0}{-1} \text{ if $t=s=0$}, \quad \mat{-1}{-1/4}{8}{1} \text{ if $t=1,\ s=0$}, \]
\[\mat{-1}{-3/4}{0}{-1} \text{ if $t=0,\ s=1$}, \quad \mat{-1}{-3/4}{8}{5} \text{ if $t=s=1$}. \]
Running $\tilde{g}$ over the double coset representatives we see that
$\tilde{g}\tilde{\beta}_t\tilde{\alpha}_s^{-1} \in \ov{K}_0$ 
implies that $\tilde{g}$ is in the double coset of 
$\ov{y}(4)\ov{w}(2^{-2})\ov{y}(4)$. 
Thus $\mathcal{W}_2 * \mathcal{U}_2$ is 
supported on $\ov{K}_0\ov{y}(4)\ov{w}(2^{-2})\ov{y}(4)\ov{K}_0$. 
Consequently $\mathcal{U}_2 * \mathcal{V} * \mathcal{U}_2 = 
\mathcal{W}_2 * \mathcal{U}_2 = \alpha \mathcal{Z}_2$
where one can compute $\alpha$ by computing 
$\mathcal{W}_2 * \mathcal{U}_2 (\tilde{g})$ where 
$\tilde{g}=\ov{y}(4)\ov{w}(2^{-2})\ov{y}(4)=:(C,\epsilon)$ and 
$\epsilon = \sigma_2(y(4), w(2^{-2}))\sigma_2(y(4)w(2^{-2}), y(4))$.
By Lemma~\ref{lem:rel1'}, 
\begin{equation*}
\mathcal{W}_2 * \mathcal{U}_2(\tilde{g}) = 
\sum_{t=0,1}\mathcal{W}_2(\tilde{g}\tilde{\beta}_t)\mathcal{U}_2(\tilde{\beta}_t^{-1}) 
= \mathcal{U}_2(\ov{w}(2^{-2}))\sum_{t=0,1}\mathcal{W}_2(\tilde{g}\tilde{\beta}_t).
\end{equation*}
Let $A_s=w(2^{-2})y(4+8s)$ and $B_t=y(-8t)w(-2^{-2})$. 
Then the matrix part of 
$\tilde{g}\tilde{\beta}_t\tilde{\alpha}_s^{-1}$ is $CB_tA_s^{-1}$ which is 
\[\mat{-1}{-1}{0}{-1} \text{ if $t=0,\ s=1$}, \qquad  \mat{1}{0}{8}{1}\text{ if $t=1,\ s=0$},\]
and the sigma-factor of 
$\tilde{g}\tilde{\beta}_t\tilde{\alpha}_s^{-1}$ is $\epsilon \eta  \sigma(C, B_tA_s^{-1})$ where
\[\eta:=\sigma(w(2^{-2}),y(4+8s))  \sigma(A_s,A_s^{-1})   
\sigma(y(-8t), w(-2^{-2}))  \sigma(B_t, A_s^{-1}).\]
$\eta$ turns out to be $-1$ when $t=0,\ s=1$ and $1$ when $t=1,\ s=0$. Thus 
\[\tilde{g}\tilde{\beta}_0 = (\mat{-1}{-1}{0}{-1},-1)\ov{w}(2^{-2})\ov{y}(4)\ov{y}(8),
\ 
\tilde{g}\tilde{\beta}_1 = (\mat{1}{0}{8}{1},1)\ov{w}(2^{-2})\ov{y}(4).\]
Thus
\[\mathcal{W}_2 * \mathcal{U}_2(\tilde{g})= 
\ov{\gamma}(\ov{w}(1))^2
(\gamma((-I,1))+1).\]
Hence $\alpha=\sqrt{2}$ and  
$\mathcal{U}_2 * \mathcal{V} * \mathcal{U}_2 = \sqrt{2}\ \mathcal{Z}_2 = \sqrt{2}\ \mathcal{V} * \mathcal{U}_2 * \mathcal{V}$.

The proof of $(4)$, $(5)$ now follows using $(1)$, $(2)$, $(3)$ above, 
Lemma~\ref{lem:rel4} and the relation
$\mathcal{U}_0 = \mathcal{T}_1 * \mathcal{U}_{1} = 
\mathcal{U}_1 * \mathcal{U}_2 * \mathcal{U}_1$. 
\end{proof}

For $n \ge 1$, define $\mathcal{R}_n := \mathcal{T}_n * \mathcal{V}$, 
$\mathcal{S}_n := \mathcal{V} * \mathcal{T}_{-n}$ and for $n \ge 2$  
$\mathcal{W}_n:= \mathcal{U}_n * \mathcal{V}$, 
$\mathcal{Y}_n:= \mathcal{V} * \mathcal{U}_n$ and 
$\mathcal{Z}_n:= \mathcal{V} * \mathcal{U}_n * \mathcal{V}$.
Note that by Lemma~\ref{lem:decomp1} and \cite[Lemma 3.4]{B-PII}, 
$\mathcal{R}_n$, $S_n$, $W_n$, $Y_n$, $Z_n$ are respectively supported on the 
$\ov{K}_0$ double cosets of $\ov{h}(2^n)\ov{y}(4)$, 
$\ov{y}(4)\ov{h}(2^{-n})$, $\ov{w}(2^{-n})\ov{y}(4)$, 
$\ov{y}(4)\ov{w}(2^{-n})$ and 
$\ov{y}(4)\ov{w}(2^{-n})\ov{y}(4)$. 
Thus it follows from Proposition~\ref{prop:sup1} that 
$\mathcal{T}_n$, $\mathcal{U}_n$ for $n \in \Z$, $\mathcal{V}$,  
$\mathcal{R}_n,\ \mathcal{S}_n$ for $n \ge 1$ and 
$\mathcal{W}_n,\ \mathcal{Y}_n,\ \mathcal{Z}_n$ for $n \ge 2$ form basis 
elements of $H(\ov{K}_0, \chi_1)$ as a vector space. Indeed it follows from
Lemma~\ref{lem:rel3} that $\mathcal{U}_1$, $\mathcal{U}_2$ and $\mathcal{V}$ 
generates $H(\ov{K}_0, \chi_1)$ as an algebra.

Let $\widehat{\mathcal{U}_1} = \frac{1}{\sqrt{2}} \mathcal{U}_1$, $\widehat{\mathcal{U}_2} = \frac{1}{\sqrt{2}}\mathcal{U}_2$ and  $\widehat{\mathcal{U}_0} = \frac{1}{2\sqrt{2}} \mathcal{U}_0$.
Using relations above, we obtain the following theorem.
\begin{thm}\label{thm:genrelchi_1} 
The Hecke algebra $H(\ov{K_0^2(8)}, \chi_1)$ is generated by 
$\widehat{\mathcal{U}_1}$, $\widehat{\mathcal{U}_2}$ and $\mathcal{V}$ 
modulo the relations:
 \begin{enumerate}
 \item $\widehat{\mathcal{U}_1}^2 = 1 + \mathcal{V}$.
 \item $\widehat{\mathcal{U}_2}^2 = 1$.
 \item $\widehat{\mathcal{U}_1} \mathcal{V} =\mathcal{V} \widehat{\mathcal{U}_1}=\widehat{\mathcal{U}_1}$.
 \item $\widehat{\mathcal{U}_2}\mathcal{V}\widehat{\mathcal{U}_2}= \mathcal{V}\widehat{\mathcal{U}_2}\mathcal{V}$.
 \end{enumerate}
\end{thm}

\subsection{The algebra $H(\ov{K_0^2(8)}, \chi_2)$}
Take $\gamma = \chi_2$, we will similarly get generators and relations for the Hecke algebra $H(\chi_2)$.

Define 
$\mathcal{Z}_1' \in H(\ov{K}_0,\chi_2)$ 
supported only on the double coset of $\ov{y}(2)\ov{w}(2^{-1})\ov{y}(2)$ such that 
$\mathcal{Z}_1'(\ov{y}(2)\ov{w}(2^{-1})\ov{y}(2))=1$. Note that 
$\ov{y}(2)\ov{w}(2^{-1})\ov{y}(2) = \ov{x}(1/2)$ and it normalizes $\ov{K}_0$.
As before we get that $\mathcal{Z}_1' * \mathcal{Z}_1' =1$.

Define $\mathcal{V}' \in H(\chi_2)$ 
supported precisely on $\ov{K}_0\ov{y}(2)\ov{K}_0$ such that 
$\mathcal{V}'(\ov{y}(2),1)= \frac{1+ \gamma((-I,1))}{\sqrt{2}}$. 
We have the following proposition.
\begin{prop}
\begin{enumerate}
 \item $\mathcal{Z}_1' * \mathcal{U}_1 * \mathcal{Z}_1' = \mathcal{V}'$.
 \item $\mathcal{U}_2 * \mathcal{Z}_1'=\mathcal{U}_2 = \mathcal{Z}_1' * \mathcal{U}_2$. 
 \item $\mathcal{U}_2 * \mathcal{U}_2 = 2 + 2 \ \mathcal{Z}_1'$.
 \item $\mathcal{U}_1 * \mathcal{U}_1 =2$.
 \item $\mathcal{U}_1 * \mathcal{Z}_1' * \mathcal{U}_1 = \sqrt{2}\ 
 \mathcal{V}' = \sqrt{2} \ \mathcal{Z}_1' * \mathcal{U}_1 * \mathcal{Z}_1'$.
 \end{enumerate}
\end{prop}
\begin{proof}
For $(1)$ observe that
\[\ov{w}(2^{-1})\ov{x}(1/2) = (\mat{0}{1/2}{-2}{-1},-1)= 
\ov{y}(2)\ov{w}(2^{-1})\ov{x}(1).\]
Thus $\mathcal{U}_1 * \mathcal{Z}_1'$ is supported precisely 
on $\ov{K}_0\ov{y}(2)\ov{w}(2^{-1})\ov{K}_0$ and it's value 
at $\ov{y}(2)\ov{w}(2^{-1}))$ is $\frac{1 -\gamma((-I,1))}{\sqrt{2}}$.
Further since
\[\ov{x}(1/2)\ov{y}(2)\ov{w}(2^{-1}) = 
(\mat{-1}{1}{-2}{1},1)= \ov{y}(2)(-x(-1),1)\]
 we get 
that $\mathcal{Z}_1' * \mathcal{Y}_1'$ is supported 
precisely at $\ov{K}_0\ov{y}(2)\ov{K}_0$ and that
\[\mathcal{Z}_1' * \mathcal{Y}_1' (\ov{y}(2))=
\frac{1 -\gamma((-I,1))}{\sqrt{2}}  \gamma((-I,1))= \frac{1 +\gamma((-I,1))}{\sqrt{2}} 
= \mathcal{V}'(\ov{y}(2)).\]
Part $(2)$ follows similarly. 
For $(3)$, $(4)$, $(5)$ we follow as in Proposition~\ref{prop:rel5}.
\end{proof}

Let $\widehat{\mathcal{U}_1} = \frac{1}{\sqrt{2}} \mathcal{U}_1$ and $\widehat{\mathcal{U}_2} = \frac{1}{\sqrt{2}}\mathcal{U}_2$.  
\begin{thm}\label{thm:genrelchi_2} 
The Hecke algebra $H(\ov{K_0^2(8)}, \chi_2)$ is generated by $\widehat{\mathcal{U}_1}$, $\widehat{\mathcal{U}_2}$ and 
$\mathcal{Z}_1'$ modulo the relations:
 \begin{enumerate}
 \item $\widehat{\mathcal{U}_1}^2 = 1$.
 \item $\widehat{\mathcal{U}_2}^2 = 1 + \mathcal{Z}_1'$.
 \item $\widehat{\mathcal{U}_2} \mathcal{Z}_1' =\widehat{\mathcal{U}_2} = 
 \mathcal{Z}_1'  \widehat{\mathcal{U}_2} $.
 \item $\widehat{\mathcal{U}_1}  \mathcal{Z}_1'  \widehat{\mathcal{U}_1}= 
 \mathcal{Z}_1'  \widehat{\mathcal{U}_1} \mathcal{Z}_1'$.
 \end{enumerate}
\end{thm}

\subsection{Local Shimura correspondence}
Loke-Savin \cite{L-S} observed an isomorphism between the Hecke algbera 
$H(\ov{K_0^2(4)}, \gamma)$ ($\gamma$ a genuine character of $\ov{K_0^2(4)}$
of order $4$) and $\PGL_2(\Q_2)$ Iwahori Hecke algebra and 
called it local Shimura correspondence. In this subsection we prove that
the Hecke algebra $H(\ov{K_0^2(8)}, \chi_i)$, $i=1,2$, is 
isomorphic to the Hecke algebra of $\GL_2(\Q_2)$ 
corresponding to $K_0(4)$ modulo scalars (here $K_0(p^n)$ denotes the subgroup of 
$\GL_2(\Z_p)$ with $(2,1)$-entry in $p^n\Z_p$). We thus verify 
local Shimura correspondence between level $8$ Hecke algebras of 
$\DSL_2(\Q_2)$ and the level $4$ Hecke algebra of $\PGL_2(\Q_2)$.

In \cite{B-P} we give generators and relations for the subalgebra 
of the Hecke algebra of $\GL_2(\Q_p)$ 
corresponding to $K_0(p^n)$ that is supported on $\GL_2(\Z_p)$ for any 
prime $p$ and natural number $n$ but do not consider the full Hecke algebra.
We will now describe the full Hecke algebra $H(\GL_2(\Q_2)//K_0(4))$. 
In this subsection we will follow the notation 
of \cite{B-P}.

For $t \in \Qs_2$, we consider the following elements of $\GL_2(\Q_2)$:\\
\[d(t) = \mat{t}{0}{0}{1},\ \ w(2^n) = \mat{0}{-1}{t}{0}, \ \ 
z(t)= \mat{t}{0}{0}{t}.\]
Note that we are also using notation $w(t)$ for denoting anti-diagonal 
elements of $\SL_2(\Q_2)$, but we hope that this abuse of notation is clear from 
the context.

We have the following lemma. 
\begin{lem}
A complete set of representatives for the double cosets of $\GL_2(\Q_2)$ mod 
$K_0(4)$ (up to central elements $z(t)$) consists of \\
$d(2^n),\ w(2^n)$ for  $n\in \Z$, $d(2^n)y(2)$ for $n \geq 0$,
$y(2)d(2^{-n})$ for $n\geq 1$ and
$y(2)w(2^n),\ w(2^n)y(2),\ y(2)w(2^n)y(2)$ for $n\geq 2$. 
\end{lem}

We also note the following decomposition of $K_0(4)$ double cosets.
\begin{lem}
\begin{small}
\begin{enumerate}
\item[(a)] For $n \ge 0$,
\[K_0(4)d(2^n)K_0(4) = \bigcup_{s\in \Z_2/2^{n}\Z_2} x(s)d(2^n)K_0(4) = 
\bigcup_{s\in \Z_2/2^{n}\Z_2} K_0(4)d(2^n)y(4s).\]
\item[(b)] For $n \ge 1$,
\[K_0(4)d(2^{-n})K_0(4) = \bigcup_{s\in \Z_2/2^{n}\Z_2} y(4s)d(2^{-n})K_0(4) = 
\bigcup_{s\in \Z_2/2^{n}\Z_2} K_0(4)d(2^{-n})x(s).\]
\item[(c)] For $n \ge 2$,
\[K_0(4)w(2^n)K_0(4) = \bigcup_{s\in \Z_2/2^{n-2}\Z_2} y(4s)w(2^n)K_0(4) = 
\bigcup_{s\in \Z_2/2^{n-2}\Z_2} K_0(4)w(2^n)y(4s).\]
\item[(d)] For $n \le 1$,
\[K_0(4)w(2^n)K_0(4) = \bigcup_{s\in \Z_2/2^{2-n}\Z_2} x(s)w(2^n)K_0(4) = 
\bigcup_{s\in \Z_2/2^{2-n}\Z_2} K_0(4)w(2^n)x(s).\]
\end{enumerate}
\end{small}
\end{lem}
Using the above lemma and since $y(2)$ normalizes $K_0(4)$ we can obtain 
decomposition of 
double cosets $K_0(4)gK_0(4)$ where $g$ varies over all the double coset 
representatives noted above.

Note that in this case Hecke algbera $H(\GL_2(\Q_2)//K_0(4))$ does not involve 
any character, so it is trivially supported on all the double cosets. 
Let $X_g$ be the characteristic function of $K_0(4)gK_0(4)$ and 
let $\mathcal{T}_n=X_{d(2^n)}$, $\mathcal{U}_n=X_{w(2^n)}$, 
$\mathcal{V}=X_{y(2)}$ and $\mathcal{Z}=X_{z(2)}$ be 
elements of the Hecke algebra $H(\GL_2(\Q_2)//K_0(4))$ 
(again note that there is a conflict of notation with the 
Hecke algebra elements of $\DSL_2(\Q_2)$ but we will see 
that the elements satisfy exactly the same relations). 
It is easy to see that $\mathcal{Z}$ is in the center and
that $\mathcal{Z}^n=X_{z(2^n)}$. 

Using \cite[Lemma 3.1]{B-P} and above decomposition 
we obtain the following relations in $H(\GL_2(\Q_2)//K_0(4))$.
\begin{lem}\label{lem:intHecke}
\begin{enumerate} 
 \item If $mn \ge 0$ then $\mathcal{T}_m*\mathcal{T}_n = \mathcal{T}_{m+n}$.
 \item For $n \le 0$, $\mathcal{U}_1 * \mathcal{T}_n = \mathcal{U}_{1+n}$ and 
 for $n \ge 0$, $\mathcal{T}_n * \mathcal{U}_1 = \mathcal{Z}^n\mathcal{U}_{1-n}$.
 \item For $n \ge 0$, $\mathcal{U}_2 * \mathcal{T}_n = \mathcal{U}_{2+n}$ and 
 for $n \le 0$, $\mathcal{T}_n * \mathcal{U}_2 = \mathcal{Z}^n\mathcal{U}_{2-n}$.
 \item For $m \ge 2$, $\mathcal{U}_1 * \mathcal{U}_m = \mathcal{Z}\mathcal{T}_{m-1}$ and
 $\mathcal{U}_m * \mathcal{U}_1 = \mathcal{Z}^m\mathcal{T}_{1-m}$.
 \item For $m \le 1$, $\mathcal{U}_2 * \mathcal{U}_m = \mathcal{Z}^2\mathcal{T}_{m-2}$ and
 $\mathcal{U}_m * \mathcal{U}_2 = \mathcal{Z}^m\mathcal{T}_{2-m}$.
\end{enumerate}
\end{lem}

We have the following proposition.
\begin{prop}\label{prop:intHecke}
\begin{enumerate}
\item $\mathcal{V} * \mathcal{V} =1$.
\item $\mathcal{U}_1 * \mathcal{U}_1 = 2\mathcal{Z}(1 + \mathcal{V})$.
\item $\mathcal{U}_1 * \mathcal{V} = \mathcal{U}_1 = \mathcal{V} * \mathcal{U}_1$.
\item $\mathcal{U}_2 * \mathcal{U}_2 = \mathcal{Z}^2$.
\item $\mathcal{U}_2 * \mathcal{V} * \mathcal{U}_2 = \mathcal{Z}\mathcal{V} * \mathcal{U}_2 * \mathcal{V}$.
\item $\mathcal{U}_0 * \mathcal{U}_0 = 4 + 2\ \mathcal{U}_0 + 4\mathcal{V}$.
\item $\mathcal{U}_0 * \mathcal{V} = \mathcal{U}_0 = \mathcal{V} * \mathcal{U}_0$.
\end{enumerate}
\end{prop}
\begin{proof}
Note that $\mathcal{V},\  \mathcal{U}_0$ are elements of the subalgebra 
supported on $\GL_2(\Z_p)$ and the relations $(1), (6), (7)$ directly follow from 
\cite[Proposition 3.10, 3.12]{B-P}. The relation $(3)$, $(4)$ and the 
braid relation $(5)$ follows easily as the above lemma.
For relation $(2)$ we use \cite[Lemma 3.2]{B-P}. 
For $s=0,1$, let $\alpha_s = x(s)w(2)$. Then
$\mathcal{U}_1*\mathcal{U}_1$ is supported on those 
$g \in \GL_2(\Q_2)$ for which there exists $s,\ t \in \{0,1\}$
such that
\[(\alpha_s\alpha_t)^{-1}g = \mat{-1/2}{s/2}{-t}{st-1/2}g \in K_0(4).\]
Checking this for $g$ as it varies over all the double coset representatives 
we get that the support is precisely on $z(2)$ and $y(2)z(2)$. Further we get 
that
\begin{equation*}
\begin{split}
\mathcal{U}_1*\mathcal{U}_1(y(2)z(2)) &=
 \sum_{s=0,1}\mathcal{U}_1(\alpha_s)\mathcal{U}_1(\alpha_s^{-1}y(2)z(2))\\
 &=
 \mathcal{U}_1(\alpha_1\mat{-1}{0}{-4}{-1}) + 
 \mathcal{U}_1(\alpha_1\mat{1}{1}{0}{1}) = 2.
\end{split}
\end{equation*}
Similarly $\mathcal{U}_1*\mathcal{U}_1(z(2))=2$. Thus we obtain $(2)$.
\end{proof}

The remaining basis elements of $H(\GL_2(\Q_2)//K_0(4))$ are 
precisely $\mathcal{T}_n* \mathcal{V}$, $\mathcal{V}*\mathcal{T}_{-n}$
for $n \ge 1$, and $\mathcal{U}_n* \mathcal{V}$, $\mathcal{V}*\mathcal{U}_{n}$
and $\mathcal{V}*\mathcal{U}_{n}*\mathcal{V}$ for $n \ge 2$.

We have the following theorem. 
\begin{thm}
The Hecke algebra $H(\GL_2(\Q_2)//K_0(4))/\langle \mathcal{Z} \rangle$ 
is generated by $\mathcal{U}_1, \ \mathcal{U}_2,\ \mathcal{V}$ with the 
defining relations:
\begin{enumerate}
 \item $\mathcal{U}_1^2 = 2(1 + \mathcal{V})$.
 \item $\mathcal{U}_2^2 = 1$.
 \item $\mathcal{U}_1 \mathcal{V} =\mathcal{V} \mathcal{U}_1=\mathcal{U}_1$.
 \item $\mathcal{U}_2\mathcal{V}\mathcal{U}_2= 
 \mathcal{V}\mathcal{U}_2\mathcal{V}$.
 \end{enumerate}
\end{thm}

\begin{cor}
We have the following isomorphism of Hecke algebras:
\[H(\ov{K_0^2(8)}, \chi_1)\cong H(\ov{K_0^2(8)}, \chi_2)
 \cong H(\GL_2(\Q_2)//K_0(4))/\langle \mathcal{Z}\rangle.
\]
\end{cor}

The Hecke algebra generators and relations described above allow a 
study of the representation theory of the maximal compact with 
$(\ov{K_0^2(8)},\gamma)$ equivariant
vectors and also the infinite dimensional genuine representations of 
$\DSL(2)$ with such vectors. We will pursue this study in 
a subsequent work.

\section{Translation of adelic to classical.}

We follow the notation as in Section 4 of \cite{B-PII}.
Let $k$ be a natural number, $M$ be odd and 
$\chi$ be an even Dirichlet character modulo $8M$.
Let $\chi_0 = \chi\kro{-1}{\cdot}^k$. 
We consider the central character $\gamma$ of $M_2$ such that 
$\gamma((-I,1)) = -i^{2k+1}$ and $\chi_1,\ \chi_2$ be the extension of $\gamma$ as 
in the previous section.

Let $A_{k+1/2}(8M, \chi_0)$ be the set of adelic cuspidal automorphic forms 
$\Phi:\DSL_2(\A)\rightarrow \C$ satisfying certain properties as 
considered by Waldspurger~\cite{Waldspurger}. 
By Gelbart-Waldspurger there is an isomorphism between
$A_{k+1/2}(8M, \chi_0) \rightarrow S_{k+1/2}(\Gamma_0(8M), \chi)$, 
$\Phi_f\leftrightarrow f$,
inducing a ring isomorphism
\[q:\mathrm{End}_{\C}(A_{k+1/2}(8M, \chi_0))\rightarrow 
\mathrm{End}_{\C}(S_{k+1/2}(\Gamma_0(8M), \chi).\] 
We will use $q$ to translate 
certain elements in $H(\ov{K_0^2(8)}, \chi_1)$ and $H(\ov{K_0^2(8)}, \chi_2)$ 
to classical operators on $S_{k+1/2}(\Gamma_0(8M))$ and 
$S_{k+1/2}(\Gamma_0(8M), \kro{2}{\cdot})$ respectively. 
Thus the classical operators so obtained satisfies the local Hecke 
algebra relations noted in the previous section. These relations 
are crucial for the results obtained in the next section. 

\begin{prop}\label{prop:trans1}
Let $\mathcal{T}_1,\ \mathcal{U}_1,\ \mathcal{U}_2,\ \mathcal{V} 
\in H(\ov{K_0^2(8)}, \chi_1)$  and 
$f\in S_{k+1/2}(\Gamma_0(8M))$.
\begin{enumerate}
\item $q(\mathcal{T}_1)(f)(z) = 2^{-(2k+1)/2} \sum_{s=0}^3 f((z+s)/4)=2^{(3-2k)/2}U_4(f)(z).$
\item $q(\mathcal{U}_1)(f)(z) = \ov{\zeta_8}\kro{-1}{M}^{k+3/2}\kro{2}{M}
2f|[W_4,\phi_{W_4}(z)]_{k+1/2}(z)$ where 
$W_4=\mat{4n}{m}{4M}{8}$ with $n,m\in \Z$ such that $8n-mM=1$ and 
 $\phi_{W_4}(z)=(2Mz+4)^{1/2}$.
\item $q(\mathcal{U}_2)(f)(z) = \ov{\zeta_8}\kro{-1}{M}^{k+3/2} 
\sum_{s=0}^{1}f|[W_8,\phi_{W_8}(z)]_{k+1/2}(z)$ where 
$W_8=\mat{16n-8mMs}{m}{16M-128Ms}{16}$ with $n,m\in \Z$ such that $16n-mM=1$ and 
 $\phi_{W_8}(z)=((4M-32Ms)z+4)^{1/2}$.
\item $q(\mathcal{V})(f)(z) =  f|[\mat{1}{0}{4M}{1},(4Mz+1)^{1/2}]_{k+1/2}(z)$.
\end{enumerate}
\end{prop}
\begin{proof}
The proof follows by similar calculations as in \cite{B-PII}. 
\end{proof}

We similarly have the following proposition. 
\begin{prop}\label{prop:trans2} Let $f\in S_{k+1/2}(\Gamma_0(8M),\kro{2}{\cdot})$.
Let $W_4,\ W_8,$ be as in the above proposition.
For $\mathcal{T}_1,\ \mathcal{U}_1,\ \mathcal{U}_2,\ 
\mathcal{Z}_1',\ \mathcal{V}' \in H(\ov{K_0^2(8)}, \chi_2)$  we have
\begin{enumerate}
\item $q(\mathcal{T}_1)(f)(z) = 2^{-(2k+1)/2} \sum_{s=0}^3 f((z+s)/4)$.
\item $q(\mathcal{U}_1)(f)(z) = \ov{\zeta_8}\kro{-1}{M}^{k+3/2}
2f|[W_4,\phi_{W_4}(z)]_{k+1/2}(z)$. 
\item $q(\mathcal{U}_2)(f)(z) = \ov{\zeta_8}\kro{-1}{M}^{k+3/2}\kro{2}{M} 
\sum_{s=0}^{1}f|[W_8,\phi_{W_8}(z)]_{k+1/2}(z)$.
\item $q(\mathcal{Z}_1')(f)(z) = f(z-\frac{1}{2})$.
\end{enumerate}
\end{prop}

\vskip2mm
From now on we consider the case of trivial character.
\noindent Define operators
$\widetilde{W}_8:= q(\widehat{\mathcal{U}_2})$ and 
$\widetilde{V}_4:= q(\mathcal{V})$ on $S_{k+1/2}(\Gamma_0(8M))$
where $\widehat{\mathcal{U}_2},\ \mathcal{V}$ are elements 
in $H(\ov{K_0^2(8)}, \chi_1)$. Note that both $\widetilde{V}_4$ and 
$\widetilde{W}_8$ are involution.
Define $\widetilde{V}'_4$ to be the conjugate of $\widetilde{V}_4$ by 
$\widetilde{W}_8$. 

We have following corollary from Theorem~\ref{thm:genrelchi_1}. 
\begin{cor}
${\widetilde{W}_8}^2 = 1,\ {\widetilde{V}_4}^2 = 1$.
\end{cor}

\begin{cor}\label{cor:plus}
$S_{k+1/2}(\Gamma_0(4M))$ is contained in the $+1$ eigenspace of 
$\widetilde{V}_4$ and $q(\mathcal{U}_1^2) =4$ on $S_{k+1/2}(\Gamma_0(4M))$.
\end{cor}
\begin{proof}
The first assertion follows directly. For the second one, observe 
that $W_4$ in Proposition~\ref{prop:trans1} is same as 
$W$ in \cite[Proposition 4.8 (2)]{B-PII}.
So for $f\in S_{k+1/2}(\Gamma_0(4M))$, we have 
$q(\mathcal{U}_1)(f) = 2\widetilde{W}_4(f)$. 
In particular, $q(\mathcal{U}_1^2)(f) =4f$ as $\widetilde{W}_4$ is an 
involution on $S_{k/2}(\Gamma_0(4M))$.
\end{proof}

\begin{lem}\label{lem:sf}
Let $\mathcal{T},\ \mathcal{T}'$ be elements of $H(\ov{K_0^2(8)}, \chi_1)$
respectively supported on the double cosets of 
$\tilde{s},\ \tilde{s}^{-1} \in \DSL_2(\Q_2)$ such that
$\mathcal{T}'(\tilde{s}^{-1}) = \overline{\mathcal{T}(\tilde{s})}$.
Then the $\mathrm{L}^2$-inner product $\langle \Phi, \mathcal{T}\Psi \rangle = 
\langle \mathcal{T}'\Phi, \Psi \rangle$ for any $\Phi,\ \Psi \in 
A_{k+1/2}(8M, \kro{-1}{\cdot}^k)$.
\end{lem}
\begin{proof}
The $\mathrm{L}^2$-inner product
\begin{small}
\begin{equation*}
\begin{split}
\langle \Phi, \mathcal{T}\Psi \rangle &=
\int_{s_{\Q}(\SL_2(\Q))\backslash\DSL_2(\A)/\mu_2} 
\Phi(h) \overline{\mathcal{T}\Psi(h)} dh\\
&=
\int_{s_{\Q}(\SL_2(\Q))\backslash\DSL_2(\A)/\mu_2} 
\Phi(h) 
\int_{\ov{K_0^2(8)}\tilde{s}\ov{K_0^2(8)}} \overline{\mathcal{T}(x)
\Psi(hx)}dx dh\\
&=
\int_{\ov{K_0^2(8)}\tilde{s}\ov{K_0^2(8)}} \overline{\mathcal{T}(x)}
\int_{s_{\Q}(\SL_2(\Q))\backslash\DSL_2(\A)/\mu_2} 
\Phi(h) \overline{\Psi(hx)}dh dx\quad \text{(Fubini)}\\
&=
\int_{s_{\Q}(\SL_2(\Q))\backslash\DSL_2(\A)/\mu_2}  
\int_{\ov{K_0^2(8)}\tilde{s}^{-1}\ov{K_0^2(8)}} 
\overline{\mathcal{T}(x^{-1})}\Phi(hx)dx\ \overline{\Psi(h)} dh\\
&=
\int_{s_{\Q}(\SL_2(\Q))\backslash\DSL_2(\A)/\mu_2}  
\int_{\ov{K_0^2(8)}\tilde{s}^{-1}\ov{K_0^2(8)}} 
\mathcal{T}'(x)\Phi(hx)dx\ \overline{\Psi(h)} dh
=\langle \mathcal{T}'\Phi, \Psi \rangle.
\end{split}
\end{equation*}
\end{small}
\end{proof}

\begin{prop}
The operators $\widetilde{W}_8,\ \widetilde{V}_4,\ \widetilde{V}'_4$ are 
self-adjoint with respect to the Petersson inner product. 
\end{prop}
\begin{proof}
By Gelbart~\cite[(3.10)]{Gelbart} for $f, g \in S_{k+1/2}(\Gamma_0(8M))$ 
the Petersson inner product $\langle f, g \rangle$ equals 
constant times the ${\mathrm{L}^2}$-inner product 
$\langle \Phi_f, \Phi_g \rangle$. In particular if $T$ is an operator on 
$S_{k+1/2}(\Gamma_0(8M))$ such that 
$T = q(\mathcal{T})$ where $\mathcal{T} \in H(\ov{K_0^2(8)}, \chi_1)$, then 
$\langle f, Tg \rangle$ equals the constant times $\langle \Phi_f, \mathcal{T}\Phi_g \rangle$.

Since the $\ov{K_0^2(8)}$ double cosets of $(y(4),1)$ and $(w(2^{-2}),1)$ equal 
respectively that of $(y(4),1)^{-1}$ and $(w(2^{-2}),1)^{-1}$, 
by Lemma~\ref{lem:sf} we are done.
\end{proof}

\subsection{Comparison with Kohnen's projection map}
Kohnen~\cite[Page 37]{Kohnen2} and later Ueda-Yamana~\cite{U-Y}
define function $P_8 (f) = f|[\xi + \xi^{-1}]_{k+1/2}$ where 
$\xi = (\mat{4}{1}{0}{4}, e^{\pi i/4})$. We have the following observation.
\begin{prop}
Let $f\in S_{k+1/2}(\Gamma_0(8M))$. Then 
$q(\mathcal{Z}_2)(f) = \kro{2}{2k+1} P_8(f)$.
\end{prop}
\begin{proof}
By \cite[equation 2.2]{U-Y}, we can write 
\[P_8(f) = e^{-(2k+1)\pi i/4} \sum_{s=0}^{1}f|[\mat{4-8Ms}{1}{-32Ms}{4}, 
(-8Msz+1)^{1/2}]_{k+1/2}.\]
Now the proof essentially follows by observing that 
$\mathcal{Z}_2$ is precisely supported on the  
double coset of 
$\ov{K_0^2(8)}\mat{1}{1/4}{0}{1}\ov{K_0^2(8)} = 
\ov{K_0^2(8)}\mat{1}{-1/4}{0}{1}\ov{K_0^2(8)} 
= \sqcup_{s=0}^{1}\mat{1}{-1/4}{8Ms}{1-2Ms}\ov{K_0^2(8)}$.
Indeed, computing as before we obtain 
\[q(\mathcal{Z}_2)(f) = \mathcal{Z}_2((\mat{1}{-1/4}{0}{1},1)) e^{(2k+1)\pi i/4}P_8(f).\]
Also it is easy to check that 
\[\mathcal{Z}_2((\mat{1}{-1/4}{0}{1},1)) = 
\mathcal{Z}_2((\mat{1}{1/4}{0}{1},1)) \ov{\gamma}(-I,-1) 
= \ov{\zeta_8}(-i^{2k+1})\] and that
$\kro{2}{2k+1}\ov{\zeta_8}(-i^{2k+1}) = e^{-(2k+1)\pi i/4}$.
\end{proof}
Now using the relation in Proposition~\ref{prop:rel5}(3) we have
\begin{cor}\label{cor:plus1}
$\frac{1}{\sqrt{2}}\kro{2}{2k+1} P_8= 
\widetilde{V}_4\widetilde{W}_8\widetilde{V}_4 = 
\widetilde{W}_8\widetilde{V}_4\widetilde{W}_8 = \widetilde{V}'_4$.
\end{cor}
Extending Kohnen's definition, Ueda-Yamana \cite{U-Y} define the plus space 
$S_{k+1/2}^{+}(8M)$ to consist of 
$f= \sum_{n=1}^{\infty}a_nq^n \in S_{k+1/2}(\Gamma_0(8M))$ such that 
$a_n=0$ for $(-1)^k n \equiv 2, 3 \pmod{4}$.
\begin{cor}\label{cor:plus2}
$S_{k+1/2}^{+}(8M)$ is the $+1$ eigenspace of 
$\widetilde{V}'_4$. 
The $-1$ eigenspace of $\widetilde{V}'_4$
consists of $f$ such that $a_n=0$ for $(-1)^k n \equiv 0, 1 \pmod{4}$. 
\end{cor}
\begin{proof}
From \cite[equation(2)]{Kohnen2}, 
$P_8(f) = \sqrt{2}\kro{2}{2k+1} (\sum_n^{(1)} a_nq^n - \sum_n^{(2)} a_n q^n)$ 
where $\sum_n^{(1)}$ resp. $\sum_n^{(2)}$ runs over $n$ with 
$(-1)^k n \equiv 0, 1 \pmod{4}$ resp. $(-1)^k n \equiv 2, 3 \pmod{4}$.
The result now follows using the above corollary.
\end{proof}
 
Consider the projection map $\wp_k$ \cite{U-Y} onto the plus space which 
take $\sum_n a_n q^n$ to $\sum_n^{(1)} a_n q^n$. 
\begin{cor}\label{cor:plus3}
If $f$ belongs to the $-1$ eigenspace of $\widetilde{V}'_4$ 
then $\wp_k(f)=0$. 
\end{cor}

\section{Minus space of $S_{k+1/2}(\Gamma_0(8M))$}\label{sec:minus}
Let $M$ be odd and square-free. In this section we shall 
define the minus space $S_{k+1/2}^{-}(8M)$ and 
show that there is an Hecke algebra isomorphism between 
$S_{k+1/2}^{-}(8M)$ 
and $S_{2k}^{\text{new}}(\Gamma_0(4M))$. We shall give a 
characterization of the minus space as common $-1$ eigenspace 
of certain operators. The method we 
employ is similar to \cite{B-PII}. The main tools that we use are 
the generators and relations
of Theorem~\ref{thm:genrelchi_1} and their translation into classical 
operators. 
We also need the operators $\widetilde{Q}_p$, $\widetilde{Q}'_p$, 
$\widetilde{W}_{p^2}$ on 
$S_{k +1/2}(\Gamma_0(2^nM))$ for $p \mid M$ and 
operators $\widetilde{Q}_2$, $\widetilde{Q}'_2$, 
$\widetilde{W}_{4}$ on $S_{k +1/2}(\Gamma_0(4M))$ that we defined 
in \cite{B-PII}.

The following proposition is crucial to our study of the minus space. 
To prove it we will use the relations in Theorem~\ref{thm:genrelchi_1} 
including the crucial braid relation (Theorem~\ref{thm:genrelchi_1}(4)).
\begin{prop}\label{prop:dsum}
 \begin{enumerate}
  \item Let $f \in S_{k+1/2}(\Gamma_0(4M)$. Then\\
  $f \in S_{k+1/2}^{+}(4M) \iff \widetilde{W}_8f = f$.
  \item $S^{+}_{k+1/2}(4M) + \widetilde{W}_4 
  S^{+}_{k+1/2}(4M) + 
  \widetilde{W}_8 \widetilde{W}_4S^{+}_{k+1/2}(4M)$ is a direct sum. 
  \item $S_{k+1/2}^{-}(4M) + \widetilde{W}_8S_{k+1/2}^{-}(4M)$ is a direct sum. 
 \end{enumerate}
\end{prop}
\begin{proof}
We first prove $(1)$. For $f \in S_{k+1/2}(\Gamma_0(4M)$ we have 
\begin{equation*}
\begin{split}
q(\widehat{\mathcal{U}_2})f &= \widetilde{W}_8f = f \Longrightarrow 
q(\mathcal{U}_1\widehat{\mathcal{U}_2})f = q(\mathcal{U}_1)f 
\Longrightarrow q(\mathcal{T}_1/\sqrt{2})f = q(\mathcal{U}_1)f\\
&\Longrightarrow q(\mathcal{U}_1\mathcal{T}_1/\sqrt{2})f = q(\mathcal{U}_1^2)f =4f
\Longrightarrow \widetilde{W}_4U_4f = 2^k f\\
&\Longrightarrow \widetilde{Q}'_2(f) = 2f
\Longrightarrow f \in S_{k+1/2}^{+}(4M).
\end{split}
\end{equation*}
The second implication follows from Lemma~\ref{lem:rel3} while the third 
and fourth follows from 
Corollary~\ref{cor:plus}. For the last part, see \cite[Section 4.3]{B-PII}. 
Now let $f \in S_{k+1/2}^{+}(4M)$. Since $f$ satisfies the plus-space 
Fourier coefficient condition, it follows from Corollary~\ref{cor:plus2} that 
$\widetilde{V}'_4(f)= f$, i.e., 
$\widetilde{W}_8\widetilde{V}_4\widetilde{W}_8(f)= 
\widetilde{V}_4\widetilde{W}_8\widetilde{V}_4(f) = f$. 
Using Corollary~\ref{cor:plus} we get that 
$\widetilde{W}_8f = f$.

We now prove $(2)$.
Let $f,\ g,\ h \in S^{+}_{k+1/2}(4M)$ be such that 
$f + q(\mathcal{U}_1)g + q(\widehat{\mathcal{U}_2}\mathcal{U}_1)h =0$
(note that $q(\mathcal{U}_1) = 2\widetilde{W}_4$ on 
$S_{k+1/2}(\Gamma_0(4M))$). 
Applying $q(\mathcal{V})$ to the above equation and using 
Corollary~\ref{cor:plus} and Theorem~\ref{thm:genrelchi_1}(3) we get
$f + q(\mathcal{U}_1)g + q(\mathcal{V}\widehat{\mathcal{U}_2}\mathcal{U}_1)h =0$.
Let $h' = q(\mathcal{U}_1)h \in S_{k+1/2}(\Gamma_0(4M))$. Subtracting the above equations we have 
$q(\widehat{\mathcal{U}_2})h' = q(\mathcal{V}\widehat{\mathcal{U}_2})h'$. 
Next applying $q(\widehat{\mathcal{U}_2})$ to the above and using 
Theorem~\ref{thm:genrelchi_1}(4) we have 
$h' = q(\mathcal{V}\widehat{\mathcal{U}_2}\mathcal{V})h'$.
As $\mathcal{V}^2 = 1$ and using Corollary~\ref{cor:plus}, we get 
$\widehat{\mathcal{U}_2}h' = h'$. Now part $(1)$ implies that 
$h' \in S_{k+1/2}^{+}(4M)$. Thus  $h' = 0$ as
$S^{+}_{k+1/2}(4M) \bigcap 
\widetilde{W}_4S^{+}_{k+1/2}(4M)=\{0\}$ 
(follows as in \cite[Proposition 6.17]{B-PII}) 
and consequently $f=g=h=0$.

For $(3)$ observe that 
$S_{k+1/2}^{-}(4M)$ is contained in the $+1$ eigenspace 
of $\widetilde{V}_4$ and  $\widetilde{W}_8S_{k+1/2}^{-}(4M)$
is contained in the $+1$ eigenspace 
of $\widetilde{V}'_4$. Let $f \ne 0$ belongs to the intersection. 
Then $\widetilde{V}_4f = f = \widetilde{V}'_4f$. Now using 
$\widetilde{V}'_4 = q(\widehat{\mathcal{U}_2}\mathcal{V}\widehat{\mathcal{U}_2}) = 
q(\mathcal{V}\widehat{\mathcal{U}_2}\mathcal{V})$
(Theorem~\ref{thm:genrelchi_1}(4))
we get $\widehat{\mathcal{U}_2}(f) = f$.
Thus by $(1)$, $f \in S_{k+1/2}^{+}(4M)\cap S_{k+1/2}^{-}(4M)$, a 
contradiction.
\end{proof}

We recall the following theorem of Ueda.
\begin{thm}(Ueda~\cite{Ueda})\label{T:Ueda}
Let $M$ be odd and square-free.
There exists an isomorphism of vector spaces
$\psi:S_{k+1/2}(\Gamma_0(8M))\to S_{2k}(\Gamma_0(4M))$ 
satisfying 
\[T_p (\psi(f))=\psi (T_{p^2}(f))\ \ \text{for all primes $p$ coprime to $2M$}.\] 
\end{thm}

We first construct the minus space at level $8$. 
In the above theorem take $M=1$. It follows using 
Proposition~\ref{prop:dsum}, Atkin-Lehner 
and dimension equality 
(see \cite[Corollary 6.1]{B-PII}) that
\begin{lem}\label{lem:Gamma1}
$\psi$ maps
$S^{+}(4) \oplus \widetilde{W}_4S^{+}(4) \oplus 
\widetilde{W}_8 \widetilde{W}_4S^{+}(4)$
isomorphically onto $S_{2k}(\Gamma_0(1))\oplus 
V(2) S_{2k}(\Gamma_0(1)) \oplus V(4) S_{2k}(\Gamma_0(1))$.
\end{lem}

Also since $S^{-}(4)$ is Hecke isomorphic to 
$S^{\mathrm{new}}_{2k}(\Gamma_0(2))$ \cite{B-P} we have 
\begin{lem}\label{lem:Gamma2}
$\psi$ maps $S^{-}(4) \oplus \widetilde{W}_8S^{-}(4)$
isomorphically onto $S^{\mathrm{new}}_{2k}(\Gamma_0(2))\oplus 
V(2) S^{\mathrm{new}}_{2k}(\Gamma_0(2))$.
\end{lem}

Let $E:= (S^{+}(4) \oplus \widetilde{W}_4S^{+}(4) \oplus 
\widetilde{W}_8 \widetilde{W}_4S^{+}(4))\ \oplus\ (S^{-}(4) \oplus \widetilde{W}_8S^{-}(4))$.
Thus $\psi$ maps $E$ Hecke isomorphically onto 
$S^{\mathrm{old}}_{2k}(\Gamma_0(4))$.

Define $S_{k+1/2}^{-}(8)$ to be the orthogonal complement of
$E$. 
\begin{thm}
$S^{-}(8)$ has a basis of eigenforms for all the operators $T_{p^2}$, $p$ odd; 
these eigenforms are also eigenfunctions under $U_{4}$. 
If two eigenforms in $S^{-}(8)$ share the same eigenvalues 
for all $T_{p^2}$ then they are a scalar multiple of each other. 
$\psi$ induces a Hecke algebra isomorphism:
\[S_{k+1/2}^{-}(8) \cong S_{2k}^{\mathrm{new}}(\Gamma_0(4)).\]
\end{thm}
\begin{proof}
The proof uses Lemma~\ref{lem:Gamma1} and \ref{lem:Gamma2}, 
Theorem~\ref{T:Ueda} and follows by the argument in \cite[Theorem 5]{B-PII}.
\end{proof}

\begin{prop}\label{prop:u4}
If $f \in S_{k+1/2}^{-}(8)$ is a Hecke eigenform for all the Hecke operators 
$T_{p^2}$, $p$ odd prime then $\widetilde{W}_8(f) = \pm f$. 

Further for any $f \in S_{k+1/2}^{-}(8)$, we have $U_4f=0$
and $\widetilde{V}_4f = -f = \widetilde{V}'_4f$.
\end{prop}
\begin{proof}
Let $f\in S_{k+1/2}^{-}(8)$ be a Hecke eigenform under all such $T_{p^2}$. 
Let $g=\widetilde{W}_8(f)$. Since $\widetilde{W}_8$ 
commutes with $T_{p^2}$, $p$ odd, we get that $g$ is an eigenform for all 
$T_{p^2}$ with the same eigenvalues as $f$. 
Since $F:=\psi(f) \in S_{2k}^{\mathrm{new}}(\Gamma_0(4))$ is a newform, 
by \cite{A-L} $\psi(g)$ is a scalar multiple of $\psi(f)$. Thus
$g$ is a a scalar multiple of $f$. Since $\widetilde{W}_8^2 =1$, we 
get the first assertion.

Further by \cite{A-L} since $F$ is a newform of level $4$, $U_2(F)=0$.
Since the Shimura lift~\cite{Shimura} $\Sh_t(f)$ for any square-free $t$ 
is also an eigenform for 
all $T_p$ with the same eigenvalues as $F$ by \cite{A-L} $\Sh_t f$ is a scalar 
multiple of $F$. Thus $\Sh_t(U_4f)= U_2(\Sh_t f) = 0$ 
for all square-free $t$ and hence we get that $U_4f=0$.

Now $0= U_4(f) = q(\mathcal{T}_1)f= 
q( \sqrt{2}\mathcal{U}_1\widehat{\mathcal{U}_2})f$. 
Since $\widetilde{W}_8(f) = \pm f$ we have 
$q(\mathcal{U}_1)f = 0$. As ${\widehat{\mathcal{U}_1}}^2 = 1 + \mathcal{V}$
(Theorem~\ref{thm:genrelchi_1}(1)) we get $\widetilde{V}_4f = -f$. 
Consequently $\widetilde{V}'_4f = -f$.

Since $S_{k+1/2}^{-}(8)$ has a basis of eigenforms under $T_{p^2}$,
it follows for all $f \in S_{k+1/2}^{-}(8)$, we have $U_4f=0$
and $\widetilde{V}_4f = -f = \widetilde{V}'_4f$.
\end{proof}

\begin{thm}\label{thm:8}
Let $f\in S_{k+1/2}(\Gamma_0(8))$. Then 
$f \in S_{k+1/2}^{-}(8) \iff \widetilde{V}_4f = -f = \widetilde{V}'_4f$. 
\end{thm}
\begin{proof}
If $f \in S_{k+1/2}^{-}(8)$ then by Proposition~\ref{prop:u4} the
conditions hold.

Conversely let $\widetilde{V}_4f = -f = \widetilde{V}'_4f$.
Since $S_{k+1/2}(\Gamma_0(4))=S^{+}(4) \oplus \widetilde{W}_4S^{+}(4) \oplus 
S^{-}(4)$ is contained in the $+1$ eigenspace of $\widetilde{V}_4$ and 
$\widetilde{W}_8(\widetilde{W}_4S^{+}(4) \oplus S^{-}(4))$ is contained 
in the $+1$ eigenspace of $\widetilde{V}'_4$ and $\widetilde{V}_4,\  
\widetilde{V}'_4$ are self-adjoint, it follows that 
$f \in S_{k+1/2}^{-}(8)$.
\end{proof}

Note that since $\widetilde{V}'_4$ is self-adjoint, we can write 
$S_{k+1/2}(\Gamma_0(8))$ as a direct sum of $+1$ and $-1$ eigenspaces 
of $\widetilde{V}'_4$. As noted in Corollary~\ref{cor:plus2}, 
$S_{k+1/2}^{+}(8)$ is the $+1$ eigenspace of $\widetilde{V}'_4$, 
let us denote by $S_{k+1/2}^{\mathrm{min}}(8)$ be the $-1$ 
eigenspace of $\widetilde{V}'_4$. In particular,
$S_{k+1/2}^{\mathrm{min}}(8)$ is the subspace of $S_{k+1/2}(\Gamma_0(8))$
consisting of $f= \sum_{n=1}^{\infty}a_nq^n$ such that 
$a_n=0$ for $(-1)^k n \equiv 0, 1 \pmod{4}$.
Further for a given newform $F$ of level dividing $4$, let 
$S_{k+1/2}(8, F)$ denotes the subspace of forms that are 
Shimura-equivalent to $F$ (i.e., forms $f$ that are eigenforms under 
$T_{p^2}$ with the same eigenvalues as $F$ under $T_p$ for all odd primes $p$).
Then we have the following simple observation.

\begin{prop}\label{prop:obs}
\begin{itemize}
 \item[(1)] $S_{k+1/2}^{+}(8) = S^{+}(4) \oplus \widetilde{W}_8 A^{+}(4) \oplus
\widetilde{W}_8 S^{-}(4)$ where $A^{+}(4) = \widetilde{W}_4 S^{+}(4)$.
 \item[(2)] Given a newform $F$ of weight $2k$ and level dividing $4$, 
 there exists a unique Shimura-equivalent form in  
 $S_{k+1/2}(8, F) \cap S_{k+1/2}^{\mathrm{min}}(8)$.
\end{itemize}
\end{prop}
\begin{proof}
Let $S:=S^{+}(4) \oplus \widetilde{W}_8 A^{+}(4) \oplus
\widetilde{W}_8 S^{-}(4)$, $R:=A^{+}(4) \oplus S^{-}(4) \oplus S^{-}(8)$.
It follows from Corollary~\ref{cor:plus} that $S \subseteq S_{k+1/2}^{+}(8)$. To prove equality 
it is enough to show that $R \cap S_{k+1/2}^{+}(8) = \{0\}$. 
Let $f+g+h$ belongs to the intersection where 
$f \in A^{+}(4),\ g\in  S^{-}(4),\ h\in S^{-}(8)$. Thus 
$\widetilde{V}'_4(f+g+h) = f+g+h$. Since 
$\widetilde{V}'_4 = \widetilde{V}_4\widetilde{W}_8\widetilde{V}_4$ and 
by Corollary~\ref{cor:plus}, it follows that 
$\widetilde{W}_8 f + \widetilde{W}_8g = f + g + 2 \widetilde{V}_4 h$. 
Since $\widetilde{V}_4$ preserves $S^{-}(8)$ and as each of the term 
in the above relation is in the direct summand, we are done. 

For $(2)$, since $T_{p^2}$ for odd prime $p$ commutes with $\widetilde{V}'_4$, 
we get that $\widetilde{V}'_4$ preserves the space $S_{k+1/2}(8, F)$. Now  
it follows from $(1)$ and Lemma \ref{lem:Gamma1}, Lemma \ref{lem:Gamma2} 
that for a weight $2k$ newform of level $1$ there are two 
Shimura-equivalent forms in the space 
$S_{k+1/2}^{+}(8)$ while for a weight $2k$ newform 
of level $2$ there is precisely one Shimura-equivalent form in 
$S_{k+1/2}^{+}(8)$. Consequently using dimension equality we obtain $(2)$.
The case of newform of level $4$ is already done in Theorem~\ref{thm:8}.
\end{proof}

We now define the minus space at level $8M$ for $M$ odd square-free.
Let $1\ne M=p_1p_2\cdots  p_k$ and 
for each $i=1, \ldots k$ define $M_{i}=M/p_i$. 
Note that by \cite[Corollary 4.3 (4)]{B-PII} 
$S_{k+1/2}(\Gamma_0(8M_i))$ is contained in the $p_i$ eigenspace 
of $\widetilde{Q}_{p_i}$. 
Now following the proof of \cite[Proposition 6.4]{B-PII} 
we obtain that
\begin{prop}
$S_{k+1/2}(\Gamma_0(8M_i)) \bigcap 
\widetilde{W}_{p_i^2}S_{k+1/2}(\Gamma_0(8M_i))=\{0\}$.
\end{prop}

Using Atkin-Lehner~\cite{A-L} and dimension equality we have 
the following.
\begin{cor} \label{C:imap}
$\psi$ maps
$S_{k+1/2}(\Gamma_0(8M_i))\oplus 
\widetilde{W}_{p_i^2}S_{k+1/2}(\Gamma_0(8M_i))$
isomorphically onto $S_{2k}(\Gamma_0(4M_i))\oplus 
V(p_i) S_{2k}(\Gamma_0(4M_i))$.
\end{cor}

Let $S^{+,\mathrm{new}}_{k+1/2}(4M)$ be the new space 
inside the Kohnen plus subspace of $S_{k+1/2}(4M)$
and $S^{-}_{k+1/2}(4M)$ as defined in \cite{B-PII}. Then 
by Proposition~\ref{prop:dsum} and Atkin-Lehner we similarly have
\begin{cor} \label{C:2map}
$\psi$ maps
$S^{+,\mathrm{new}}_{k+1/2}(4M) \oplus 
\widetilde{W}_4S^{+,\mathrm{new}}_{k+1/2}(4M) \oplus 
\widetilde{W}_8\widetilde{W}_4S^{+,\mathrm{new}}_{k+1/2}(4M)$ 
isomorphically onto 
$S^{\mathrm{new}}_{2k}(\Gamma_0(M)) \oplus 
V(2) S^{\mathrm{new}}_{2k}(\Gamma_0(M))
\oplus V(4) S^{\mathrm{new}}_{2k}(\Gamma_0(M))$.
\end{cor}

\begin{cor} \label{C:3map}
$\psi$ maps
$S^{-}_{k+1/2}(4M) \oplus 
\widetilde{W}_8S^{-}_{k+1/2}(4M)$ 
isomorphically onto 
$S^{\mathrm{new}}_{2k}(\Gamma_0(2M)) \oplus 
V(2) S^{\mathrm{new}}_{2k}(\Gamma_0(2M))$.
\end{cor}

We note the following observation. 
\begin{remark}
Since $S^{-}_{k+1/2}(4M)$ is contained in the $+ 1$ eigenspace
of $\widetilde{V}_4$, $\widetilde{W}_8S^{-}_{k+1/2}(4M)$ 
is contained in the  $+ 1$ eigenspace
of $\widetilde{V}'_4$ and hence is contained inside $S_{k+1/2}^{+}(8M)$. 
In \cite{U-Y}, Ueda-Yamana defined a newspace inside $S_{k+1/2}^{+}(8M)$ and 
proved that it is Hecke isomorphic to $S^{\mathrm{new}}_{2k}(\Gamma_0(2M))$. 
Using the above corollary and following Proposition~\ref{prop:obs} we 
see that the plus newspace identified by \cite{U-Y} is the space 
$\widetilde{W}_8S^{-}_{k+1/2}(4M)$. 
Note that $\widetilde{V}'_4$ does 
not preserve the space $S_{k+1/2}(\Gamma_0(4M))$ and so we do not 
expect Fourier coefficient condition for $S^{-}_{k+1/2}(4M)$, as 
it was also observed in \cite{B-PII}. 
\end{remark}

Now let $B_i=S_{k+1/2}(\Gamma_0(8M_i)) \oplus 
\widetilde{W}_{p_i^2}S_{k+1/2}(\Gamma_0(8M_i))$, 
$i=1,\ldots k$. Define
\begin{equation*}
\begin{split}
E=\sum_{i=1}^k B_i \oplus &
(S^{+,\mathrm{new}}_{k+1/2}(4M) \oplus 
\widetilde{W}_4S^{+,\mathrm{new}}_{k+1/2}(4M) \oplus 
\widetilde{W}_8\widetilde{W}_4S^{+,\mathrm{new}}_{k+1/2}(4M))\\
&\oplus S^{-}_{k+1/2}(4M) \oplus 
\widetilde{W}_8S^{-}_{k+1/2}(4M).
\end{split}
\end{equation*}
\begin{prop}
Under $\psi$ the space $E$ maps isomorphically onto the old space 
$S_{2k}^{\mathrm{old}}(\Gamma_0(4M))$.
\end{prop}
\begin{proof}
This follows from Corollary~\ref{C:imap}, \ref{C:2map} and 
and \ref{C:3map} and from the decomposition
\begin{equation*}
\begin{split}
S_{2k}^{\mathrm{old}}(\Gamma_0(4M))&=\left( 
\sum_{i=1}^{k}S_{2k}(\Gamma_0(4M_i))\oplus 
V(p_i)S_{2k}(\Gamma_0(4M_i)) \right) \\
&\oplus 
\left(S_{2k}^{\mathrm{new}}(\Gamma_0(M))\oplus 
V(2)S_{2k}^{\mathrm{new}}(\Gamma_0(M))\oplus 
V(4)S_{2k}^{\mathrm{new}}(\Gamma_0(M))\right)\\
&\oplus
S^{\mathrm{new}}_{2k}(\Gamma_0(2M)) \oplus 
V(2) S^{\mathrm{new}}_{2k}(\Gamma_0(2M))
\end{split}
\end{equation*}
\end{proof}

We now define the minus space to be the orthogonal complement of $E$, 
\[S^{-}_{k+1/2}(4M):=E^{\perp}.\]

\begin{thm}
The space $S^{-}_{k+1/2}(8M)$ has a basis of eigenforms for 
all the operators $T_{q^2}$ where $q$ is an odd prime
satisfying $(q,M)=1$. Under
$\psi$, the space  $S^{-}_{k+1/2}(8M)$ maps isomorphically onto the 
space $S^{\mathrm{new}}_{2k}(\Gamma_0(4M))$. If 
two forms in $S^{-}_{k+1/2}(8M)$ have the same eigenvalues for
all the operators $T_{q^2}$, $(q,2M)=1$, then they are same up to a scalar 
factor. Moreover, $S^{-}_{k+1/2}(8M)$ has strong multiplicity one 
property in the full space of level $8M$. 
\end{thm}

We give the characterization of our minus space.
We have the following proposition.
\begin{prop}\label{prop:u4M}
Let $f \in S_{k+1/2}^{-}(8M)$ be a Hecke eigenform for all the Hecke operators 
$T_{q^2}$, $q$ prime and $q$ coprime to $2M$. Then for any prime $p$ dividing 
$M$, 
$\widetilde{W}_(p^2) = \pm f$, $\widetilde{W}_8(f) = \pm f$.
Moreover, $U_{p^2}(f)=-p^{k-1}\lambda(p)f$ and $U_{4}(f)=0$ where 
$\lambda(p)= \pm 1$.

Consequently, for any $f \in S_{k+1/2}^{-}(8M)$ we have
$\widetilde{Q}_p(f) = -f = \widetilde{Q}'_p(f)$ for all primes 
$p$ dividing $M$ and 
$\widetilde{V}_4f = -f = \widetilde{V}'_4f$.
\end{prop}
\begin{proof}
The proof follows similar to the proof of Proposition~\ref{prop:u4M}
and proofs of \cite[Proposition 6.12, 6.13, 6.14]{B-PII}.
\end{proof}

\begin{thm}\label{thm:8M}
Let $f\in S_{k+1/2}(8M)$. Then $f\in S^{-}_{k+1/2}(8M)$ if and only if
$\widetilde{Q}_p(f) = -f = \widetilde{Q}'_p(f)$ for every prime 
$p$ dividing $M$ and 
$\widetilde{V}_4(f) = -f = \widetilde{V}'_4(f)$.
\end{thm}
\begin{proof}
One side implication follows from Proposition~\ref{prop:u4M}. 
For the converse, we use Corollary~\ref{cor:plus}, that 
$S_{k+1/2}(\Gamma_0(8M/p))$ is contained in the $p$ eigenspace 
of $\widetilde{Q}_{p}$ for all $p$ dividing $M$ and that the 
operators are self-adjoint.
\end{proof}

\begin{cor}\label{cor:minus1}
If $f = \sum_{n=1}^{\infty} a_n q^n \in S^{-}_{k+1/2}(8M)$ then 
$a_n=0$ for $(-1)^k n \equiv 0, 1 \pmod{4}$. 

In particular the projection map $\wp_k$ is identically zero on 
the minus space $S^{-}_{k+1/2}(8M)$.
\end{cor}
\begin{proof}
The proof follows from the above theorem and Corollary~\ref{cor:plus1} and 
\ref{cor:plus2}. 
\end{proof}

\begin{remark}\label{rem:MMR}
The above corollary contradicts some of the results of \cite{M-M-R}. 
In particular in Section 3 of their paper 
they assert that the projection map $\wp_k$ on the
newforms at level $8M$ is itself, i.e. their newspace at level 
$8M$ (which corresponds to $S^{\mathrm{new}}_{2k}(\Gamma_0(4M))$) 
satisfies the plus space condition. However our results above and 
the example below presents a contrary picture : if 
$f = \sum_{n=1}^{\infty} a_n q^n$ is in the newspace at level $8M$ then 
$a_n=0$ for $(-1)^k n \equiv 0, 1 \pmod{4}$, i.e., $\wp_k(f)=0$.
\end{remark}

\begin{remark}
We note that Theorem~\ref{thm:8} and \ref{thm:8M} are analogous 
to \cite[Theorem 9]{B-P} in the integral weight scenario. 
Indeed \ref{thm:8M} can be restated as 
$f\in S^{-}_{k+1/2}(8M)$ if and only if
$\widetilde{Q}_p(f) = -f = \widetilde{Q}'_p(f)$ for every prime 
$p$ dividing $M$ and 
$q(\mathcal{U}_1)f = 0 = \widetilde{W}_8q(\mathcal{U}_1)\widetilde{W}_8(f)$.
\end{remark}

\begin{remark}
We note that the decomposition of the space $S_{k+1/2}(\Gamma_0(8M))$ 
is completely analogous to that of $S_{2k}(\Gamma_0(4M))$ when we look
at it through the local Hecke algebra. We illustrate this in the case 
$M=1$.
\begin{equation*}
\begin{split}
S_{2k}(\Gamma_0(4)) &= (S_{2k}(\Gamma_0(1)) \oplus 
 q(\mathcal{U}_1)S_{2k}(\Gamma_0(1)) \oplus  q(\mathcal{U}_2)S_{2k}(\Gamma_0(1)))\\
 &\oplus (S^{\mathrm{new}}_{2k}(\Gamma_0(2))
 \oplus 
 q(\mathcal{U}_2)S^{\mathrm{new}}_{2k}(\Gamma_0(2)))
 \oplus S^{\mathrm{new}}_{2k}(\Gamma_0(4)).
\end{split}
\end{equation*}
In the above $\mathcal{U}_1,\ \mathcal{U}_2$ are elements in 
the Hecke algebra $H(\GL_2(\Q_2)//K_0(4))$ 
coming from the double cosets of $\smallmat{0}{-1}{2}{0},\ \smallmat{0}{-1}{4}{0}$ 
respectively. 
Also it follows from \cite{B-P} that 
$q(\mathcal{U}_2)q(\mathcal{U}_1)S_{2k}(\Gamma_0(1)) = 
q(\mathcal{U}_1)S_{2k}(\Gamma_0(1))$. 

Now let's look at the space $S_{k+1/2}(\Gamma_0(8M))$. We have 
\begin{equation*}
\begin{split} 
S_{k+1/2}(\Gamma_0(8)) &= (A^{+}(4) \oplus q(\mathcal{U}_1)A^{+}(4)
\oplus q(\mathcal{U}_2)A^{+}(4))\\
&\oplus (S^{-}(4) \oplus q(\mathcal{U}_2)S^{-}(4))
\oplus S^{-}(8).
\end{split}
\end{equation*}
Here $\mathcal{U}_1,\ \mathcal{U}_2$ are elements in 
the Hecke algebra $H(K_0^2(8),\chi_1)$ coming from 
$w(2^{-1}),\ w(2^{-2})$ respectively. 
Recall from \cite{B-PII} that 
$A^{+}(4) = \widetilde{W}_4S^{+}(4) = q(\mathcal{U}_1)S^{+}(4)$. 
Further by Proposition~\ref{prop:dsum},
$q(\mathcal{U}_2)q(\mathcal{U}_1)A^{+}(4) = 
q(\mathcal{U}_2)S^{+}(4) = S^{+}(4) = q(\mathcal{U}_1)A^{+}(4)$.
\end{remark}

\begin{ex}
The space $S_{3/2}(\Gamma_0(152))$ is $8$-dimensional and 
there are four primitive Hecke eigenforms of weight $2$ and level 
dividing $76$, namely $F_{19}$ of level $19$, $G_{38}$, $H_{38}$ 
of level $38$ and $P_{76}$ of level $76$. 
We have
\begin{small}
 \[S_{3/2}(\Gamma_0(152)) = S_{3/2} (152, F_{19}) 
\oplus S_{3/2}(152, G_{38})
\oplus S_{3/2}(152, H_{38})\oplus S_{3/2}(152, K_{76}).\]
\end{small}
We compute the Shimura decomposition \cite{Purkait}.
As we would expect from the above remark, $S_{3/2}(152, F_{19})$ is $3$-dimensional
space and is spanned by 
\begin{equation*}
\begin{split}
f_1 &= q+ q^5 - 2q^6 - q^9 - q^{17} + 2q^{25} + 2q^{30} + 2q^{42} -
    3q^{45} + O(q^{50})\\
f_2 &= q^4 - 2q^{11} - 2q^{16} + 2q^{19} + q^{20} - 2q^{24} + 3q^{28} + 2q^{35}
    - q^{36} + O(q^{40})\\
f_3 &= q^7 - q^{11} - 2q^{16} + q^{19} + 2q^{28} + q^{35} - 2q^{39} - q^{43} +
    2q^{44} - q^{47} + O(q^{50}),
\end{split}
\end{equation*}
$S_{3/2}(152, G_{38})$ is $2$-dimensional
space and is spanned by
\begin{equation*}
\begin{split}
g_1 &= q - 2q^5 + q^6 + 2q^9 - q^{17} - q^{25} - 3q^{26} - 4q^{30} +
    3q^{38} + 5q^{42} + O(q^{50})\\
g_2 &= q^4 + q^7 - q^{16} - 2q^{20} - 3q^{23} + q^{24} - q^{28} + 2q^{36} +
    q^{39} + 2q^{47} + O(q^{50}),
\end{split}
\end{equation*}
$S_{3/2}(152, H_{38})$ is $2$-dimensional
space and is spanned by
\begin{equation*}
\begin{split}
h_1 &= q^2 + 2q^{10} - 3q^{13} - q^{14} - 2q^{18} - q^{21} + 2q^{22} + q^{29} +
O(q^{30})\\
h_2 &= q^3 - q^8 + q^{12} - q^{19} - q^{27} - q^{32} - 2q^{40} + q^{48} + O(q^{50})
\end{split}
\end{equation*}
and $S_{3/2}(152, K_{76})$ is $1$-dimensional
space and is spanned by
\[k_1 = q^2 - q^{10} - q^{14} + q^{18} + 2q^{21} - q^{22} - 2q^{29} - 2q^{33} -
    q^{34} + 2q^{37} + q^{38} - 2q^{41} + O(q^{50}).\]
The Kohnen plus space $S_{3/2}^{+}(152)$ is $4$-dimensional and it is 
spanned by $\{f_2, f_3, g_2, h_2 \}$.
We further note that $S_{3/2}(76, F_{19})$ is $2$-dimensional and spanned by 
$\{f_1+f_3, f_2-f_3\}$ and $S_{3/2}^{-}(76)$ is $2$-dimensional spanned 
by $\{g_1-g_2, h_1-h_2\}$.
The minus space at level $152$, $S_{3/2}^{-}(152)$ 
is $1$-dimensional spanned by $k_1$ and is 
Shimura equivalent to $K_{76}$. Note that $k_1$ satisfies the Fourier 
coefficient condition as noted in Corollary \ref{cor:minus1}.
\end{ex}

We finally look at the minus space of level $8M$ with character $\kro{2}{\cdot}$. 
Following Ueda~\cite{Ueda2}, we define  
$\widetilde{\tau}_8 : S_{k+1/2}(\Gamma_0(8M)) \longrightarrow
S_{k+1/2}(\Gamma_0(8M), \kro{2}{\cdot})$ given by the action
$|[\mat{8a}{b}{8Mc}{8d}, \kro{Mc}{d}8^{1/4}(i(Mcz +d))^{1/2}]_{k+1/2}$ 
where $a,\ b,\ c,\ d$ are such that $8ad-Mbc =1$ and $b \equiv d \equiv 1  \pmod{8}$. 
The above action is independent of the choice of 
$a,\ b,\ c,\ d$ satisfying the conditions. It is routine to check that 
$\widetilde{\tau}_8 \Delta_0(8M) \widetilde{\tau}_8^{-1} = 
\Delta_0(8M, \kro{2}{\cdot})$ and that $\widetilde{\tau}_8^{2} = 1$ on 
$S_{k+1/2}(\Gamma_0(8M))$. Further we check that 
$\widetilde{\tau}_8$ commutes with Hecke operators $T_{p^2}$ for all $p$ odd,
giving a Hecke isomorphism from $S_{k+1/2}(\Gamma_0(8M))$ to 
$S_{k+1/2}(\Gamma_0(8M), \kro{2}{\cdot})$.

Define the minus space $S^{-}_{k+1/2}(8M, \kro{2}{\cdot}) 
:= \widetilde{\tau}_8 S^{-}_{k+1/2}(8M)$. 
It follows that $S^{-}_{k+1/2}(8M, \kro{2}{\cdot})$ maps Hecke isomorphically onto 
$S^{\mathrm{new}}_{2k}(\Gamma_0(4M))$.  
Further $S^{-}_{k+1/2}(8M, \kro{2}{\cdot})$ has a similar characterization :
$g \in S^{-}_{k+1/2}(8M, \kro{2}{\cdot})$ if and only if 
$\widetilde{\tau}_8\widetilde{Q}_p\widetilde{\tau}_8^{-1}(g) = -g = 
\widetilde{\tau}_8\widetilde{Q}'_p\widetilde{\tau}_8^{-1}(g)$ for every prime 
$p$ dividing $M$ and 
$\widetilde{\tau}_8\widetilde{V}_4\widetilde{\tau}_8^{-1}(g) = -g = 
\widetilde{\tau}_8\widetilde{V}'_4\widetilde{\tau}_8^{-1}$.

Let $\mathcal{Z}_1' \in H(\ov{K_0^2(8)}, \chi_2)$.
\begin{prop}
The action of $\widetilde{\tau}_8\widetilde{V}_4\widetilde{\tau}_8^{-1}$ 
equates the action of $q(\mathcal{Z}_1')$ on 
$S_{k+1/2}(\Gamma_0(8M), \kro{2}{\cdot})$. 
In particular $S^{-}_{k+1/2}(8M, \kro{2}{\cdot})$ is contained in the 
$-1$ eigenspace of $q(\mathcal{Z}_1')$.
\end{prop}
\begin{proof}
By Proposition~\ref{prop:trans1} and \ref{prop:trans2}, 
$\widetilde{V}_4$ acts by $|[\mat{1}{0}{4M}{1}, 1]_{k+1/2}$ and
$q(\mathcal{Z}_1')$ acts by
$|[\mat{1}{-1/2}{0}{1}, 1]_{k+1/2}$. It is easy to see that  
\[\widetilde{\tau}_8^{-1}[\mat{1}{0}{4M}{1}, 1]\widetilde{\tau}_8\cdot 
[\mat{1}{1/2}{0}{1}, 1] = [\mat{*}{*}{C_1}{D_1}, (C_1z +D_1)^{1/2}]\]
where $C_1 = 32Mx^2$ and $D_1 = -Mwy+16Mx^2+4Myx+8zx$. As 
$D_1 \equiv 1 \pmod{4}$, 
$\epsilon_{D_1} =1$. Further $\kro{C_1}{D_1}= \kro{2M}{D_1} = \kro{2}{D_1}\kro{8xz}{M} = \kro{2}{D_1}$. 
Thus the right hand side belongs to $\Delta_0(8M, \kro{2}{\cdot})$ and we are done.
\end{proof}

\begin{remark}
We note that when $M=1$, the action of $\widetilde{\tau}_8$ is same as 
that of $|[\mat{0}{-1}{8}{0}, 8^{1/4}(-iz)^{1/2}]_{k+1/2}$ and we can check that 
$\widetilde{\tau}_8\widetilde{W}_8\widetilde{\tau}_8^{-1}= 
q(\widehat{\mathcal{U}_1})$ (Proposition~\ref{prop:trans2}) on $S_{k+1/2}(\Gamma_0(8), \kro{2}{\cdot})$, 
recall $\widehat{\mathcal{U}_1} \in H(\ov{K_0^2(8)}, \chi_2)$ is an involution. Thus, 
$g \in S^{-}_{k+1/2}(8, \kro{2}{\cdot})$ if and only if 
$q(\mathcal{Z}_1')(g) = -g = 
q(\widehat{\mathcal{U}_1}\mathcal{Z}_1'\widehat{\mathcal{U}_1})(g)$.
\end{remark}


\begin{thebibliography}{}

\bibitem{A-L} A.\ O.\ L.\ Atkin and J.\ Lehner,
{\em Hecke operators on $\Gamma_0(m)$},
Math.\ Ann.\ {\bf 185} (1970), 134--160.

\bibitem{B-P} E.\ M.\ Baruch and S.\ Purkait,
{\em Hecke algebras, new vectors and newforms on 
$\Gamma_0(m)$}, 
Math.\ Zeit.\ {\bf 287} (2017), no. 1-2, 705-733..

\bibitem{B-PII} E.\ M.\ Baruch and S.\ Purkait,
{\em Newforms of half-integral weight: the minus space counterpart},
arXiv:1609.06481v2, 1--44.

\bibitem{Gelbart} S.\ Gelbart,
{\em Weil's representation and the spectrum of the metaplectic group},
Lecture Notes in Math.\ {\bf 530}, Springer, Berlin, 1976.

\bibitem{Kohnen1} W.\ Kohnen,
{\em Modular forms of half-integral weight on $\Gamma_0(4)$}, 
Math.\ Ann.\ {\bf 248} (1980), 249--266. 

\bibitem{Kohnen2} W.\ Kohnen,
{\em Newforms of half-integral weight},
J.\ Reine Angew.\ Math.\ {\bf 333} (1982), 32--72.

\bibitem{L-S} H.\ Y.\ Loke and G.\ Savin,
{\em Representations of the two-fold central
extension of $\SL_2(\Q_2)$},
Pacific J.\ Math.\ {\bf 247} (2010), 435--454.

\bibitem{M-R-V} M.\ Manickam, B.\ Ramakrishnan, T.\ Vasudevan, 
{\em On the theory of newforms of half-integral weight},
J.\ Number Theory {\bf 34} (1990), 210--224.

\bibitem{M-M-R} M.\ Manickam, J.\ Meher, B.\ Ramakrishnan, 
{\em Theory of newforms of half-integral weight},
Pacific J.\ Math.\ {\bf 274} (2015), 125--139.

\bibitem{Niwa} S.\ Niwa,
{\em On Shimura's trace formula}, 
Nagoya Math.\ J.\ {\bf 66} (1977), 183--202.

\bibitem{Purkait} S.\ Purkait, 
{\em On Shimura's decomposition}, 
Int.\ J.\ Number Theory {\bf 9} (2013), 1431-1445.


\bibitem{Shimura} G.\ Shimura,
{\em On modular forms of half integral weight}, 
Ann.\ of Math.\ {\bf 97} (1973), 440--481.

\bibitem{Ueda} M.\ Ueda, 
{\em The decomposition of the spaces of cusp forms
of half-integral weight and trace formula
of Hecke operators},
J.\ Math.\ Kyoto Univ.\ {\bf 28-3} (1988), 505--555.

\bibitem{Ueda2} M.\ Ueda, 
{\em On twisting operators and newforms of half-integral weight},
Nagoya Math.\ J.\ {\bf 131} (1993), 135--205.

\bibitem{U-Y} M.\ Ueda and S.\ Yamana,
{\em On newforms for Kohnen plus spaces},
Math.\ Z.\ {\bf 264} (2010), 1--13.

\bibitem{Ueda-un} M.\ Ueda, 
{\em Newforms of half-integral weight in the case of level $2^m$}, 
(2011), 1--7.

\bibitem{Waldspurger} J.-L.\ Waldspurger,
{\em Sur les coefficients de Fourier des formes modulaires de 
poids demi-entier}, J.\ Math.\ Pures Appl.\ (9) {\bf 60} (1981), 375--484. 
\end{thebibliography}
\end{document}